\documentclass[11pt,a4paper,twoside]{amsart}
\usepackage[activeacute,english]{babel}
\usepackage[latin1]{inputenc}
\usepackage{amsfonts}
\usepackage{amsmath,amsthm}
\usepackage{amssymb}
\usepackage{graphics}
\usepackage{tikz}

\newcommand{\BB}{{\mathcal B}}
\newcommand{\CC}{{\mathcal C}}
\newcommand{\DD}{{\mathcal D}}
\newcommand{\EE}{{\mathcal E}}

\newcommand{\GG}{{\mathcal G}}
\newcommand{\HH}{{\mathcal H}}

\newcommand{\LL}{{\mathcal L}}
\newcommand{\MM}{{\mathcal M}}

\newcommand{\RR}{{\mathcal R}}
\newcommand{\SSS}{{\mathcal S}}
\newcommand{\TT}{{\mathcal T}}
\newcommand{\VV}{{\mathcal V}}

\newcommand{\N}{{\mathbb N}}
\newcommand{\R}{{\mathbb R}}
\newcommand{\Z}{{\mathbb Z}}

\newcommand{\wit}{\widetilde}
\newcommand{\wih}{\widehat}

\newcommand{\supp}{{\operatorname{supp}}}
\newcommand{\diam}{{\operatorname{diam}}}

\newcommand{\Lip}{{\operatorname{Lip}}}
\newcommand{\dist}{{\operatorname{dist}}}
\newcommand{\Top}{{\operatorname{Top}}}
\newcommand{\Tree}{{\operatorname{Tree}}}
\newcommand{\Child}{{\operatorname{Ch}}}
\newcommand{\Trees}{{\operatorname{Trs}}}
\newcommand{\Stop}{{\operatorname{Stp}}}

\newtheorem{theorem}{Theorem}[section]
\newtheorem{lemma}[theorem]{Lemma}
\newtheorem{corollary}[theorem]{Corollary}

\newtheorem{claim}[theorem]{Claim}
\newtheorem{question}[theorem]{Question}

\allowdisplaybreaks[1]

\date{}

\begin{document}

\title[Variation for singular integrals on uniformly rectifiable sets]
{$L^p$-estimates for the variation for singular integrals on uniformly rectifiable sets}

\author{Albert Mas and Xavier Tolsa}
%\address{Xavier Tolsa, ICREA/Universitat Aut\`onoma de Barcelona, Bellaterra 08913, Catalonia}

\address{Albert Mas. Departament de Matem\`atica Aplicada I,
ETSEIB, Universitat Polit\`ecnica de Catalunya. Avda. Diagonal 647, 08028 Barcelona (Spain)}
\email{amasblesa@gmail.com}

\address{Xavier Tolsa. Instituci\'{o} Catalana de Recerca i Estudis Avan\c{c}ats (ICREA) and Departament de Ma\-te\-m\`a\-ti\-ques, Universitat Aut\`onoma de Bar\-ce\-lo\-na, Catalonia}
\email{xtolsa@mat.uab.cat}

\thanks{A.M. was supported by the {\em Juan de la Cierva} program JCI2012-14073 (MEC, Gobierno de Espa\~na), ERC grant 320501 of the European Research Council (FP7/2007-2013), MTM2011-27739 and  MTM2010-16232 (MICINN, Gobierno de Espa\~na), and IT-641-13 (DEUI, Gobierno Vasco).
X.T. was supported by the ERC grant 320501 of the European Research
Council (FP7/2007-2013) and partially supported by MTM-2010-16232,  MTM-2013-44304-P (MICINN, Spain), 2014-SGR-75 (Catalonia),
and by Marie Curie ITN MAnET (FP7-607647).}
\begin{abstract}
The $L^p$ ($1<p<\infty$) and weak-$L^1$ estimates for the variation for Calder\'on-Zygmund operators with smooth odd kernel on uniformly rectifiable measures are proven. The $L^2$ boundedness and the corona decomposition method are two key ingredients of the proof.
\end{abstract}
\maketitle

\section{Introduction}
This article is devoted to obtain $L^p$ ($1<p<\infty$) and weak-$L^1$ estimates for the variation for Calder\'on-Zygmund operators with smooth odd kernel with respect to  uniformly rectifiable measures. As a matter of fact, we prove that if the $L^2$ estimate holds then the $L^p$ and weak-$L^1$ estimates follow; the results in \cite{MT} deal with the $L^2$ case.

Regarding the Calder\'on-Zygmund operators, given $1\leq n<d$ integers, in this article we consider kernels $K:\R^d\setminus\{0\}\to\R$ such that $K(-x)=-K(x)$ for all $x\neq0$ ($K$ is odd) and
\begin{equation}\label{4eq333}
|K(x)|\leq \frac{C}{|x|^{n}},\quad|\partial_{x_i}K(x)|\leq
\frac{C}{|x|^{n+1}}\quad\text{and}\quad|\partial_{x_i}\partial_{x_j}K(x)|\leq
\frac{C}{|x|^{n+2}}
\end{equation}
for all
$x=(x_1,\ldots,x_d)\in\R^d\setminus\{0\}$ and all $1\leq i,j\leq d$, where and $C>0$ is some constant. The growth estimate on the second derivatives required in \eqref{4eq333} comes from the fact that it is also assumed in \cite[Theorem 1.3 and Corollary 4.2]{MT}, which are used in this article (see Theorem \ref{weakL1 thm smooth}). We should mention that this growth estimate is usually required in what concerns to $L^2$ boundedness of singular integral operators and uniformly rectifiable measures, see for example \cite{DS1,David-Semmes,MTproc,MT,Tolsa4}. However, in Theorem \ref{teopra} below we consider more general kernels.

Given a Radon measure $\mu$ in $\R^d$, $f\in L^1(\mu)$ and $x\in\R^d$, we set
\begin{equation}\label{eqtkmu**}
T_{\epsilon}^\mu f(x)\equiv T_{\epsilon}(f\mu)(x):=\int_{|x-y|>\epsilon}K(x-y)f(y)\,d\mu(y),
\end{equation}
and we denote $T^\mu_* f(x)=\sup_{\epsilon>0}|T^\mu_\epsilon f(x)|$, $\TT=\{T_\epsilon\}_{\epsilon>0}$ and
$\TT^\mu =\{T^\mu_\epsilon\}_{\epsilon>0}$. Given $\rho>2$ and $f\in L_{loc}^1(\mu)$, the $\rho$-variation operator acting on $\TT^\mu f = \{T^\mu_\epsilon f\}_{\epsilon>0}$ is defined as
\begin{equation}\label{eqvar**}
\begin{split}
(\VV_\rho\circ\TT^\mu)f(x):=\sup_{\{\epsilon_m\}}
\bigg(\sum_{m\in\Z}
|T^\mu_{\epsilon_m}f(x)-T^\mu_{\epsilon_{m+1}}f(x)|^{\rho}\bigg)^{1/\rho}
\end{split}
\end{equation}
where the pointwise supremum is taken over all the non-increasing sequences of positive numbers $\{\epsilon_m\}_{m\in\Z}$.

Concerning the notion of uniform rectifiability, recall that a Radon measure $\mu$ in $\R^d$ is called $n$-rectifiable if there exists a countable family of $n$-dimensional $C^1$ submanifolds $\{M_i\}_{i\in\N}$ in $\R^d$ such that $\mu(E\setminus\bigcup_{i\in\N}M_i)=0$ and $\mu\ll\HH^n$, where $\HH^n$ stands for the $n$-dimensional Hausdorff measure.
Moreover, $\mu$ is said to be $n$-dimensional Ahlfors-David
regular, or simply $n$-AD regular, if there exists some constant $C>0$
such that $$C^{-1}r^n\leq\mu(B(x,r))\leq Cr^n$$ for all $x\in\supp\mu$
and $0<r\leq\diam(\supp\mu)$.  Note that if $\diam(\supp\mu)<+\infty$ then $\mu(\R^d)<\infty$ and so the condition $\mu(B(x,r))\leq Cr^n$ in the definition of AD regularity actually holds for all $r>0$.
Finally, one says that $\mu$ is uniformly $n$-rectifiable if it is $n$-AD regular and there exist $\theta,M>0$ so that, for each
$x\in\supp\mu$ and $0<r\leq\diam(\supp\mu)$, there is a Lipschitz mapping $g$ from
the $n$-dimensional ball $B^n(0,r)\subset\R^n$ into $\R^d$ such that
$\Lip(g)\leq M$ and
$$\mu\big(B(x,r)\cap g(B^n(0,r))\big) \geq \theta r^n,$$
where $\Lip(g)$ stands for the Lipschitz constant of $g$. In particular, uniform rectifiability implies rectifiability.
A set $E\subset\R^d$ is called $n$-rectifiable (or uniformly $n$-rectifiable) if $\HH^n|_E$ is $n$-rectifiable (or uniformly
$n$-rectifiable, respectively).

We are ready now to state our main result. In the statement $M(\R^d)$ stands for the Banach space of finite real Radon measures in $\R^d$ equipped with the total variation norm.

\begin{theorem}\label{teopri}
Let $\mu$ be a uniformly $n$-rectifiable measure in $\R^d$. Let $K$ be an odd kernel satisfying \eqref{4eq333} and, for $\rho>2$, consider the associated variation operator defined in \eqref{eqvar**}.
Then $$ \VV_\rho\circ\TT^\mu:L^p(\mu)\to L^p(\mu)\quad(1<p<\infty)
\quad\text{and}\quad\VV_\rho\circ\TT:M(\R^d)\to L^{1,\infty}(\mu)$$ are bounded operators. In particular,
$\VV_\rho\circ\TT^\mu: L^{1}(\mu)\to L^{1,\infty}(\mu)$ is bounded.
\end{theorem}

The variation operator has been studied in different contexts during the last years, being probability, ergodic theory, and harmonic analysis three areas where variational inequalities turned out to be a powerful tool to prove new results or to enhace already known ones (see for example \cite{Bourgain,JKRW-ergodic,JSW, JW,Lepingle,MTX,OSTTW}, and the references therein). Inspired by the
results on variational inequalities for Calder\'on-Zygmund operators in $\R^n$ like \cite{CJRW-Hilbert,CJRW-singular integrals}, in \cite{MTproc} we began our study of such type of inequalities when one replaces the underlying space $\R^n$ and its associated Lebesgue measure by some reasonable measure in $\R^d$, being the Hausdorff measure on a Lipschitz graph a first natural candidate. In this regard, Theorem \ref{teopri} should be considered as a natural generalisation of variational inequalities for Calder\'on-Zygmund operators in $\R^n$ from a geometric measure-theoretic point of view.

A big motivation to prove Theorem \ref{teopri} is its connection to the so called David-Semmes problem regarding the Riesz transform and rectifiability. Given a Radon measure $\mu$ in $\R^d$, one defines the
$n$-dimensional Riesz transform of a function $f\in L^1(\mu)$ by
$R^\mu f(x)=\lim_{\epsilon\searrow0}R^\mu_\epsilon f(x)$
(whenever the limit exists), where
$$R^\mu_\epsilon f(x)=\int_{|x-y|>\epsilon}\frac{x-y}
{|x-y|^{n+1}}\,f(y)\,d\mu(y),\qquad x\in\R^d.$$
Note that the kernel of the
Riesz transform is the vector
$(x^1,\ldots,x^d)/|x|^{n+1}$ (so, in this case, the kernel $K$ in \eqref{4eq333} is vectorial).
We also use the notation $\RR^\mu f(x):=\{R_{\epsilon}^\mu f(x)\}_{\epsilon>0}$ and, as usual, we define the maximal operator $R^\mu_* f(x)=\sup_{\epsilon>0}|R^\mu_\epsilon f(x)|$.

G. David and S. Semmes asked more than twenty years ago the following question, which is still open (see, for example, \cite[Chapter 7]{Pajot}):
\begin{question}\label{ques David-Semmes}
Is it true that
an $n$-dimensional AD regular measure $\mu$ is uniformly
$n$-rectifiable if and only if $R_*^\mu$ is bounded in $L^2(\mu)$?
\end{question}

By \cite{DS1}, the ``only if'' implication of this question above is already known to hold. Also in \cite{DS1}, G. David and S. Semmes gave a positive answer to the other implication  if one replaces the $L^2$ boundedness of $R^\mu_*$ by the $L^2$ boundedness of $T^\mu_*$ for a wide class of odd kernels $K$. In the case $n=1$ the ``if'' implication was proved in \cite{MMV} using the notion of curvature of measures. Later on, the same implication was answered affirmatively for $n=d-1$ in the work \cite{NToV} by combining quasiorthogonality arguments with some variational estimates which use the maximum principle derived from the fact that the Riesz kernel is (a multiple) of the gradient of the fundamental solution of the Laplacian in $\R^d$ when $n=d-1$. Question \ref{ques David-Semmes} is still open for the general case $1< n<d-1$. However, thanks to Theorem \ref{teopri} and \cite[Theorem 2.3]{MT} we get the following corollary, which characterizes uniform rectifiability in terms of variational inequalities for the Riesz transform and more general Calder\'on-Zygmund operators.
\begin{corollary}\label{4main theorem}
Let $\mu$ be an $n$-dimensional AD regular Radon measure in $\R^d$. Then, the following are equivalent:
\begin{itemize}
\item[$(a)$] $\mu$ is uniformly $n$-rectifiable,
\item[$(b)$] for any odd kernel $K$ as in \eqref{4eq333} and any $\rho>2$,  $\VV_\rho\circ\TT^\mu$ is bounded in $L^p(\mu)$ for all $1<p<\infty$, and from $L^1(\mu)$ into $L^{1,\infty}(\mu)$,
\item[$(c)$] for some $\rho>0$, $\VV_\rho\circ\RR^\mu$ is bounded in $L^2(\mu)$.
\end{itemize}
\end{corollary}

Comparing Corollary \ref{4main theorem} to Question \ref{ques David-Semmes},
note that the corollary asserts that if we replace the $L^2(\mu)$ boundedness of $R_*^\mu$ by the stronger
assumption that $\VV_\rho\circ\RR^\mu$ is bounded in $L^2(\mu)$, then $\mu$ must be uniformly rectifiable.
On the other hand, the corollary claims that the variation for singular integral operators with any odd kernel
satisfying \eqref{4eq333}, in particular for the $n$-dimensional Riesz transforms, is bounded in $L^p(\mu)$ for all $1<p<\infty$ and it is of weak-type $(1,1)$, which is a stronger conclusion than the one derived from an affirmative answer to Question \ref{ques David-Semmes}.

The proof of $(c)\Longrightarrow(a)$ in Corollary \ref{4main theorem} is not as hard as the converse implications. Essentally, a combination of the arguments in \cite{Tolsa4} with the fact that, in a sense,
$\VV_\rho\circ\RR^\mu$ controls $R^\mu_*$ does the job (see \cite{MT}).
Theorem \ref{teopri} is used to prove that $(a)\Longrightarrow(b)$ in Corollary \ref{4main theorem}, the corresponding result in \cite{MT} was only proved for $p=2$. Theorem \ref{teopri} allows us to get it in full generality, completing the whole picture on variation for singular integrals and uniform rectifiability. As far as we know, neither the $L^p$ estimates with $1<p<\infty$ nor the weak-$L^1$ estimate for $\VV_\rho\circ\TT^\mu$ on uniform rectifiable measures $\mu$ were known, except for the case $p=2$ treated in \cite{MT} and the case where $1<p<\infty$ but $\supp\mu$ is a Lipschitz graph with slope strictly smaller than $1$, solved in \cite{M}. Let us stress that from the latter result one can not easily deduce the $L^p$ estimates on uniformly rectifiable measures (as in the standard situation in Calder\'on-Zygmund theory), basically because the good-$\lambda$ method does not work properly for $\VV_\rho\circ\TT$. To avoid this obstacle, our method relies on the corona decomposition technique combined with some ideas from the Lipschitz case in \cite{M} and from \cite{CJRW-Hilbert} and \cite{MTX} to deal with variational inequalities, as well as the $L^2$ result from \cite{MT}.

Finally we wish to remark that the same techniques used to prove Theorem \ref{teopri} yield the following
result, which applies to more general Calder\'on-Zymund operators. See Section \ref{sec5} for the proof. 

\begin{theorem}\label{teopra}
For $1\leq n <d$, let $\mu$ be a uniformly $n$-rectifiable measure in $\R^d$. Let $K:\R^d\times \R^d\setminus\{(x,y):x=y\}\to\R$ be a kernel such that
$$
|K(x,y)|\leq \frac{C}{|x-y|^{n}} \quad \mbox{ for all $x\neq y\in\R^d$, }$$
and
$$
|K(x,y) - K(x',y)| + |K(y,x) - K(y,x')| \leq \frac{C\,|x-x'|}{|x-y|^{n+1}}$$
for all $x,x',y\in\R^d$  with $|x-x'|\leq \frac12|x-y|$.
For $\epsilon>0$, denote
$$T_{\epsilon}^\mu f(x)\equiv T_{\epsilon}(f\mu)(x):=\int_{|x-y|>\epsilon}K(x,y)f(y)\,d\mu(y).$$
Let
$\TT^\mu f = \{T^\mu_\epsilon f\}_{\epsilon>0}$ and let $(\VV_\rho\circ\TT^\mu)$ be defined  as in
\eqref{eqvar**}.
If $ \VV_\rho\circ\TT^\mu$ is bounded in $L^2(\mu)$, then it is also bounded in $L^p(\mu)$ for $1<p<\infty$
and from $L^1(\mu)$ to $L^{1,\infty}(\mu)$. Also,  $ \VV_\rho\circ\TT$ is bounded from $M(\R^d)$ to $L^{1,\infty}(\mu)$.
\end{theorem}

 \section*{Acknowledgement}
 We are very grateful to the anonymous referee for his/her careful reading of the paper and  for his/her comments and suggestions that improved its readability.
%In particular, we would like to acknowlegde his/her detailed report and the efford that he/she has taken to deeply study our first version of the paper.

\section{Preliminaries and auxiliary results}

\subsection{Notation and terminology}
As usual, in the paper the letter `$C$' (or `$c$') stands
for some constant which may change its value at different
occurrences, and which quite often only depends on $n$ and $d$.
Given two families of constants $A(t)$ and $B(t)$, where $t$ stands for all the explicit or implicit parameters involving $A(t)$ and $B(t)$,
the notation $A(t)\lesssim B(t)$ ($A(t)\gtrsim B(t)$) means that
there is some fixed constant $C$ such that $A(t)\leq CB(t)$ ($A(t)\geq CB(t)$) for all $t$, with $C$ as above. Also, $A(t)\approx B(t)$ is equivalent to $A(t)\lesssim B(t)\lesssim A(t)$.

Throughout all the paper we assume that $1\leq n<d$ are integers and that $\mu$ is an $n$-dimensional AD-regular measure in $\R^d$.
Given a bounded Borel set $A\subset \R^d$ and $f\in L^1_{loc}(\mu)$, we write the mean of $f$ on $A$ with respect to $\mu$ as follows:
$$m_Af:=\frac{1}{\mu(A)}\int_A f\,d\mu.$$

We consider the centered maximal Hardy-Littlewood operator:
$$
\MM f(x)=\sup_{r>0}m_{B(x,r)}|f|.$$
This is known to be bounded in $L^p(\mu)$, for $1<p\leq\infty$, and from $M(\R^d)$ to $L^{1,\infty}(\mu)$.
For $1\leq q<\infty$, we also set
$$\MM_q f:=\MM(|f|^q)^{1/q}.$$
This is
 bounded in $L^p(\mu)$, for $q<p\leq\infty$, and from $L^q(\mu)$ to $L^{q,\infty}(\mu)$.

Given $0\leq a<b$, consider the closed annulus
$$A(x,a,b):=\overline{B(x,b)}\setminus B(x,a).$$
Given $k\in\Z$, set $$I_k:=[2^{-k-1},2^{-k}).$$ One defines the {\em short} and {\em long variation} operators $\VV^\SSS_\rho\circ\TT^\mu$ and $\VV^\LL_\rho\circ\TT^\mu$, respectively, by
\begin{equation*}
\begin{split}
&(\VV^\SSS_\rho\circ\TT^\mu)f(x):=\sup_{\{\epsilon_m\}}
\bigg(\sum_{k\in\Z}\,\sum_{\epsilon_m,\epsilon_{m+1}\in I_k}
|T^\mu_{\epsilon_m}f(x)-T^\mu_{\epsilon_{m+1}}f(x)|^{\rho}\bigg)^{1/\rho},\\
&(\VV^\LL_\rho\circ\TT^\mu)f(x):=\sup_{\{\epsilon_m\}}\bigg(\sum_{\begin{subarray}{c}m\in\Z:\,\epsilon_m\in I_j,\,\epsilon_{m+1}\in I_k\\ \text{ for some }j<k\end{subarray}}|T^\mu_{\epsilon_m}f(x)-T^\mu_{\epsilon_{m+1}}f(x)|^{\rho}\bigg)^{1/\rho},
\end{split}
\end{equation*}
where, in both cases, the pointwise  supremum is taken over all the non-increasing sequences of positive numbers $\{\epsilon_m\}_{m\in\Z}$. Given a finite Borel measure $\nu$ in $\R^d$, one defines $(\VV^\SSS_\rho\circ\TT)\nu(x)$ and $(\VV^\LL_\rho\circ\TT)\nu(x)$ similarly.
For convenience of notation, given $0<\epsilon\leq\delta$ we set
\begin{equation}\label{T_de}
T_{\delta,\epsilon}:=T_\delta-T_\epsilon\quad\text{and $T_{\delta,\epsilon}^\nu$ analogously}.
\end{equation}

Let $\varphi_\R:[0,+\infty)\to[0,+\infty)$ be a non-decreasing $\CC^2$ function with $\chi_{[4,\infty)}\leq\varphi_\R\leq
\chi_{[1/4,\infty)}$ and set $\varphi_\epsilon(x) =\varphi_\R\bigl(|x|^2/\epsilon^2\bigr)$.
%, and $\varphi_\epsilon^\delta(x) =\varphi_\epsilon(x) -\varphi_\delta(x) .$
We define
\begin{equation}\label{eqlast*}
\begin{split}
T_{\varphi_\epsilon}\nu(x) &:=  \int \varphi_\epsilon(x-y)K(x-y)\,d\nu(y)\quad\text{for }x\in\R^d
\end{split}
\end{equation}
(with $K(x-y)$ replaced by $K(x,y)$ if $K$ is as in Theorem \ref{teopra}).
Finally, write
$\TT_\varphi:=\{T_{\varphi_\epsilon}\}_{\epsilon>0}$.
Compare the operator in \eqref{eqlast*} to
$$T_{\epsilon}\nu(x) = \int \chi_\epsilon(x-y)K(x-y)\,d\nu(y),$$
where $\chi_\epsilon(\cdot):=\chi_{(1,\infty)}(|\cdot|/\epsilon)$,
and the family $\TT_\varphi$ to $\TT$.

\subsection{Dyadic lattices}\label{dyadic lattice}
For the study of the uniformly rectifiable measures we will use the ``dyadic cubes'' built by G. David in \cite[Appendix 1]{David-LNM} (see also \cite[Chapter 3 of Part I]{David-Semmes}). These dyadic cubes are not true cubes, but they play this role with respect to a given $n$-dimensional AD regular Radon measure $\mu$, in a sense.

Let us explain
which are the precise results and properties of this lattice of dyadic cubes.
Given an $n$-dimensional AD regular Radon measure $\mu$ in $\R^d$ (for simplicity, here we may assume that $\diam(\supp\mu)=\infty$), for each $j\in\Z$ there exists a family $\DD^\mu_j$ of Borel subsets of $\supp\mu$ (the dyadic cubes of the $j$-th generation) such that:
\begin{itemize}
\item[$(a)$] each $\DD^\mu_j$ is a partition of $\supp\mu$, i.e.\ $\supp\mu=\bigcup_{Q\in \DD^\mu_j} Q$ and $Q\cap Q'=\varnothing$ whenever $Q,Q'\in\DD^\mu_j$ and
$Q\neq Q'$;
\item[$(b)$] if $Q\in\DD^\mu_j$ and $Q'\in\DD^\mu_k$ with $k\leq j$, then either $Q\subset Q'$ or $Q\cap Q'=\varnothing$;
\item[$(c)$] for all $j\in\Z$ and $Q\in\DD^\mu_j$, we have $2^{-j}\lesssim\diam(Q)\leq2^{-j}$ and $\mu(Q)\approx 2^{-jn}$;
\item[$(d)$] there exists $C>0$ such that, for all $j\in\Z$, $Q\in\DD^\mu_j$, and $0<\tau<1$,
\begin{equation}\label{small boundary condition}
\begin{split}
\mu\big(\{x\in Q:\, &\dist(x,\supp\mu\setminus Q)\leq\tau2^{-j}\}\big)\\&+\mu\big(\{x\in \supp\mu\setminus Q:\, \dist(x,Q)\leq\tau2^{-j}\}\big)\leq C\tau^{1/C}2^{-jn}.
\end{split}
\end{equation}
This property is usually called the {\em small boundaries condition}.
From (\ref{small boundary condition}), it follows that there is a point $z_Q\in Q$ (the center of $Q$) such that $\dist(z_Q,\supp\mu\setminus Q)\gtrsim 2^{-j}$ (see \cite[Lemma 3.5 of Part I]{David-Semmes}).
\end{itemize}
We set $\DD^\mu:=\bigcup_{j\in\Z}\DD_j^\mu$.
Given a cube $Q\in\DD_j^\mu$, we say that its side length is $2^{-j}$, and we denote it by $\ell(Q)$. Notice that $\diam(Q)\leq\ell(Q)$. For $\lambda>1$, we also write
$$\lambda Q = \bigl\{x\in \supp\mu:\, \dist(x,Q)\leq (\lambda-1)\,\ell(Q)\bigr\}.$$
We denote
\begin{equation}\label{defbq}
B_Q:=B(z_Q,c_1\ell(Q)),
\end{equation}
where $c_1\geq1$ is some big constant which will be chosen below, depending on other parameters.

Let $P(Q)$ denote the cube in $\DD^\mu_{j-1}$ which contains $Q$ (the {\em parent} of $Q$), and set
\begin{equation*}
\begin{split}
& \Child(Q):=\{Q'\in\DD^\mu_{j+1}:\,Q'\subset Q\},\\
&V(Q):=\{Q'\in\DD^\mu_j:\, \dist(Q',Q)\leq C_1\ell(Q)\}
\end{split}
\end{equation*}
for some constant $C_1>0$ big enough ($\Child(Q)$ are the {\em children} of $Q$, and $V(Q)$ stands for the {\em vicinity} of $Q$).
Notice that $P(Q)$ is a cube from $\DD^\mu$ but $ \Child(Q)$ and $V(Q)$ are collections of cubes from $\DD^\mu$. It is not hard to show that the number of cubes in $ \Child(Q)$ and $V(Q)$ is bounded by some constant depending only on $n$ and the AD regularity constant of $\mu$, and on $C_1$ in the case of the vicinitiy.

The following assumptions will be used in the sequel: $c_1$ in \eqref{defbq} is big enough so that $$Q\cup B_{Q'}\subset B_Q\text{ for all }Q'\in\Child(Q)$$ and $C_1$ is big enough so that
$$B_Q\cap\supp\mu\subset\textstyle{\bigcup_{Q'\in V(Q)}}Q'.$$

Finally, we write
$$I_Q:=I_j=[\ell(Q)/2,\ell(Q)).$$

\subsection{The corona decomposition}\label{5ss corona decomposition}

Given an $n$-dimensional AD regular Radon measure $\mu$ on $\R^d$ consider the dyadic lattice $\DD^\mu$ introduced in Subsection \ref{dyadic lattice}.
Following \cite[Definitions 3.13 and 3.19 of Part I]{David-Semmes}, one says that $\mu$ admits a corona decomposition if, for each $\eta>0$ and $\theta>0$, one can find a triple $(\BB,\GG,\Trees)$, where $\BB$ and $\GG$ are two subsets of $\DD^\mu$ (the ``bad cubes'' and the ``good cubes'') and $\Trees$ is a family of subsets $S\subset\GG$ (that we will call {\em trees}), which satisfy the following conditions::
\begin{itemize}
\item[$(a)$] $\DD^\mu=\BB\cup\GG$\quad and\quad$\BB\cap\GG=\varnothing.$
\item[$(b)$] $\BB$ satisfies a Carleson packing condition, i.e., $\sum_{Q\in\BB:\, Q\subset R}\mu(Q)\lesssim\mu(R)$ for all $R\in\DD^\mu$.
\item[$(c)$] $\GG=\biguplus_{S\in\Trees}S$, i.e., any $Q\in\GG$ belongs to only one $S\in\Trees$.
\item[$(d)$] Each $S\in\Trees$ is {\em coherent}. This means that each $S\in\Trees$ has a unique maximal element $Q_S$ which contains all other elements of $S$ as subsets, that $Q'\in S$ as soon as $Q'\in\DD^\mu$ satisfies $Q\subset Q'\subset Q_S$ for some $Q\in S$, and that if $Q\in S$ then either all of the children of $Q$ lie in $S$ or none of them do (recall that if $Q\in\DD^\mu_j$, the {\em children} of $Q$ is defined as the collection of cubes $Q'\in\DD^\mu_{j+1}$ such that $Q'\subset Q$).
\item[$(e)$] The maximal cubes $Q_S$, for $S\in\Trees$, satisfy a Carleson packing condition. That is, $\sum_{S\in\Trees:\, Q_S\subset R}\mu(Q_S)\lesssim\mu(R)$ for all $R\in\DD^\mu$.
\item[$(f)$] For each $S\in\Trees$, there exists an $n$-dimensional Lipschitz graph $\Gamma_S$ with constant smaller than $\eta$ such that $\dist(x,\Gamma_S)\leq\theta\,\diam(Q)$ whenever $x\in2Q$ and $Q\in S$ (one can replace ``$x\in2Q$'' by ``$x\in c_2Q$'' for any constant $c_2\geq2$ given in advance, by \cite[Lemma 3.31 of Part I]{David-Semmes}).
\end{itemize}

It is shown in \cite{DS1} (see also \cite{David-Semmes}) that if $\mu$ is uniformly rectifiable then it
admits a corona decomposition for all parameters $k>2$ and $\eta,\theta>0$. Conversely,
the existence of a corona decomposition for a single set of parameters $k>2$ and $\eta,\theta>0$ implies that $\mu$ is uniformly rectifiable.

We set
$$\Top_\GG=\{Q_S:S\in\Trees\}\quad\mbox{ and }\quad \Top =\Top_\GG\cup\BB.$$
If $\mu$ is uniformly rectifiable, then, by the properties $(b)$ and $(e)$ above, for all $R\in\DD^\mu$ we have
$$\sum_{Q\in\Top:\, Q\subset R}\mu(Q)\lesssim\mu(R).$$

If $R\in S$ for some $S\in\Trees$, we denote by $\Tree(R)$ the set of cubes $Q\in S$ such that $Q\subset R$ (the {\em tree} of $R$).
Otherwise, that is, if $R\in\BB$, we set $\Tree(R):=\{R\}$. Finally,  $\Stop(R)$ stands for
 the set of cubes $Q\in\BB\cup(\GG\setminus \Tree(R))$ such that $Q\subset R$ and $P(Q)\in \Tree(R)$ (the {\em stopping} cubes relative to $R$), so actually $Q\subsetneq R$.
Notice that if $R\in\BB$, then we have $\Stop(R)=\Child(R)$.

\subsection{Auxiliary results}

The following lemma follows directly from \cite[Lemma 2.14]{Tolsa-llibre}    (see also \cite[Lemma 2.2]{M} for the case of Lipschitz graphs).
\begin{lemma}[Calder\'{o}n-Zygmund decomposition]\label{lema CZ}
Let $\mu$ be a compactly supported uniformly $n$-rectifiable measure in $\R^d$. For every positive measure $\nu\in M(\R^d)$ with compact support and every $\lambda>2^{d+1}\|\nu\|/\|\mu\|$, the following hold:
\begin{itemize}
\item[$(a)$] There exists a finite or countable collection of cubes $\{Q_j\}_j$ centered at $\supp\,\nu$ which are almost disjoint, that is $\sum_j\chi_{Q_j}\leq C$ (with $C$ depending only on $d$), and a function $f\in L^1(\mu)$ such that
\begin{gather}
\nu(Q_j)>2^{-d-1}\lambda\mu(2Q_j),\label{lema CZ 1}\\
\nu(\eta Q_j)\leq2^{-d-1}\lambda\mu(2\eta Q_j)\quad\text{for }\eta>2,\label{lema CZ 2}\\
\nu=f\mu\text{ in }\R^d\setminus\Omega\text{ with }|f|\leq\lambda\;\,\mu\text{-a.e},\text{ where }\Omega={\textstyle\bigcup_j}Q_j.\label{lema CZ 3}
\end{gather}
\item[$(b)$]For each $j$, let $R_j:=6Q_j$ and denote $w_j:=\chi_{Q_j}\left(\sum_k\chi_{Q_k}\right)^{-1}$. Then, there exists a family of functions $\{b_j\}_j$ with $\supp b_j\subset R_j$ and with constant sign satisfying
\begin{gather}
\int b_j\,d\mu=\int w_j\,d\nu,\label{lema CZ 4}\\
\| b_j\|_{L^\infty(\mu)}\mu(R_j)\leq C\nu(Q_j),\text{ and}\label{lema CZ 5}\\
{\textstyle\sum_j}|b_j|\leq C_0\lambda,\quad
\text{where $C_0$ is some absolute constant.}\label{lema CZ 6}
\end{gather}
\end{itemize}
\end{lemma}

\vspace{2mm}
Let us remark that the cubes in the preceding lemma are ``true cubes'', i.e.\ they do not belong to $\DD^\mu$.

Notice that from \eqref{lema CZ 2} it follows that  $4.5 Q_j\cap\supp\mu\neq\varnothing$, which
implies that
\begin{equation}\label{eqad43}
\mu(\eta Q_j)\approx \ell(\eta Q_j)^n\qquad \mbox{for $\eta>5$ such that $\ell(\eta Q_j)\lesssim \diam(\supp\mu)$}.
\end{equation}
Additionally, if we assume that
\begin{equation}\label{eqsupp}
\supp\,\nu \subset {\mathcal U}_{\diam(\supp\mu)}(\supp\mu),
\end{equation}
where ${\mathcal U}_t(A)$ stands for the $t$-neighborhood of $A$, then
we infer that $\ell(Q_j)\leq C\diam(\supp\mu)$, for all $j$ and for some absolute constant $C$. Otherwise, for $C$ big enough we would deduce that
$$\supp\mu\cup\supp\nu\subset 2Q_j,$$ and thus $\mu(2Q_j)=\|\mu\|$ and $\nu(Q_j)\leq\|\nu\|$, so by \eqref{lema CZ 1}
$$\|\nu\|>2^{-d-1}\lambda\|\mu\|,$$
but this contradicts the choice of $\lambda$. In particular, under the assumption \eqref{eqsupp}, we infer that
\begin{equation}\label{xcvbvcx}
\mu(R_j)\approx\ell(R_j)^n\approx\ell(Q_j)^n.
\end{equation}

We will need the following version of the dyadic Carleson embedding theorem.

\begin{theorem}[Dyadic Carleson embedding theorem] \label{teocarleson}
\index{Carleson's embedding theorem}
\index{dyadic Carleson embedding theorem}
Let $\mu$ be a Radon measure on $\R^d$.
Let $\DD$ be some dyadic lattice from $\R^d$ and let $\{a_Q\}_{Q\in\DD}$ be a family of
non-negative numbers. Suppose that for every cube $R\in\DD$ we have
\begin{equation} \label{pack}
\sum_{Q\in\DD:Q\subset R} a_Q \leq c_3\,\mu(R).
\end{equation}
Then every family of non-negative numbers $\{\gamma_Q\}_{Q\in\DD}$ satisfies
\begin{equation}\label{eqcarleson1}
\sum_{Q\in\DD} \gamma_Q\,a_Q\leq c_3\int \sup_{Q\ni x}\gamma_Q\,d\mu(x).
\end{equation}
Also, for $p\in (1,\infty)$, if $f\in L^p(\mu)$,
\begin{equation}\label{eqcarleson2}
\sum_{Q\in\DD} |m_Q f|^p \,a_Q\leq c\,c_3\|f\|_{L^p(\mu)}^p,
\end{equation}
where $m_Qf =\int_Qf\,d\mu/\mu(Q)$ and
$c$ is an absolute constant.
\end{theorem}

In the preceding theorem, the lattice $\DD$ can be, for example, either the usual dyadic lattice of $\R^d$ or,
in the case when $\mu$ is
AD-regular, the lattice of cubes associated with $\mu$. For the proof of this classical result, see
\cite[Theorem 5.8]{Tolsa-llibre}, for example.

We say that $\CC\subset\DD$ is a Carleson family of cubes if
$$
\sum_{Q\in\CC:Q\subset R} \mu(Q) \leq c_3\,\mu(R)\qquad\mbox{for all $R\in\DD$.}$$
By \eqref{eqcarleson2}, it follows that for such a family $\CC$ and any $f\in L^p(\mu)$,
$$\sum_{Q\in\CC} |m_Q f|^p \,\mu(Q)\leq c\,c_3\|f\|_{L^p(\mu)}^p.
$$

\begin{lemma}\label{lemcarleson}
Let $\nu\in M(\R^d)$ be a positive measure with compact support and $\lambda>2^{d+1}\|\nu\|/\|\mu\|$. Consider cubes $\{Q_j\}_j$ and $\{R_j\}_j$ as in
 {\em Lemma \ref{lema CZ}}. Denote
$$\nu_b:=\sum_j\left(w_j\nu-b_j\mu\right),$$
where the $b_j$'s satisfy \eqref{lema CZ 4}, \eqref{lema CZ 5} and \eqref{lema CZ 6}, and
$w_j:=\chi_{Q_j}\left(\sum_k \chi_{Q_k}\right)^{-1}$.
Let $\CC\subset \DD^\mu$ be a family of cubes and $\{a_S\}_{S\in\CC}$ be a family of
non-negative numbers such that
\begin{equation} \label{pack'}
\sum_{S\in\CC:S\subset R} a_S \leq c_3\,\mu(R).
\end{equation}
For each $S\in\CC$ consider the ball $B_S$ given by $(\ref{defbq})$, so it is centered
on $S$, $S\subset B_S$ and $r(B_S)\approx \ell(S)$. Suppose that there exists some constant $\tilde c >0$ such that for each $S\in \CC$, the ball $\tilde c B_S$ contains some
cube $R_j$.
Then, for every $p\in(1,\infty)$,
\begin{equation}\label{eqdk211}
\sum_{S\in\CC} \left(\frac{|\nu_b|(B_S)}{\ell(S)^n}\right)^p a_S \lesssim\,\lambda^{p-1}\,\|\nu\|
\end{equation}
and
\begin{equation}\label{eqdk212}
\sum_{S\in\CC} \left(\frac{\nu(B_S)}{\ell(S)^n}\right)^p a_S \lesssim\,\lambda^{p-1}\,\|\nu\|,
\end{equation}
with the implicit constants depending on $p$, $c_3$, and $\tilde c$.
\end{lemma}

In particular, this lemma applies to the case when $a_S=1$ for all $S\in\CC$ and $\CC$ is a Carleson family satisfying the additional conditions stated in the lemma.

\begin{proof} First we will show \eqref{eqdk211}.
By \eqref{eqcarleson1} in Theorem \ref{teocarleson}, one gets
\begin{equation}\label{eqdecompp1}
\sum_{S\in\CC} \left(\frac{|\nu_b|(B_S)}{\ell(S)^n}\right)^p a_S\leq
c_3\int \left(\sup_{S\ni x}\frac{|\nu_b|(B_S)}{\ell(S)^n}\right)^p\,d\mu(x).
\end{equation}
Write
$$\wit \nu_b = \sum_j w_j\nu \quad \mbox{ and }\quad \wit g = \sum_j b_j,$$
so that, for every $S\in\CC$,
$$|\nu_b|(B_S)\leq \wit \nu_b(B_S) + \int_{B_S}\wit g\,d\mu.$$
Note that the measure $\wit \nu_b$ and the functions $b_j$, $\widetilde g$ are positive because $\nu$ is assumed to be  a positive measure.
By \eqref{eqdecompp1} then we have
\begin{equation}\label{eqdecomp94}
\sum_{S\in\CC} \left(\frac{|\nu_b|(B_S)}{\ell(S)^n}\right)^p a_S \lesssim
\int \left(\sup_{S\ni x}\frac{\wit \nu_b(B_S)}{\ell(S)^n}\right)^p\,d\mu(x)
+ \int \left(\sup_{S\ni x}m_{B_S}\wit g\right)^p\,d\mu(x),
\end{equation}
where $m_{B_S} \wit g = \int_{B_S}\wit g\,d\mu/\mu(B_S)$ and we have taken into account that
$\mu(B_S)\approx\ell(S)^n$.

To deal with the last integral on the right hand side of \eqref{eqdecomp94} we use
 the non-centered maximal Hardy-Littlewood operator defined by
$$\wit\MM f(x) = \sup_{B\ni x} \frac1{\mu(B)}\int_B|f|\,d\mu,$$
where the supremum is taken over all the balls which contain $x$ and whose center lies on $\supp\mu$.
Recalling that $\wit\MM$ is bounded in $L^p(\mu)$, and using that $\|\wit g\|_{L^\infty(\mu)}\leq c\,\lambda$ (by
\eqref{lema CZ 6}) and $\|\wit g\|_{L^1(\mu)}\leq c\,\|\nu\|$ (by \eqref{lema CZ 5}), we obtain
\begin{equation}\label{eqgg423}
\int \left(\sup_{S\ni x}m_{B_S}\wit g\right)^p\,d\mu(x)\leq c\int (\wit\MM\wit g)^{\,p}\,d\mu \leq c\int \wit g^{\,p}\,d\mu
\leq c\lambda^{p-1}\,\int \wit g\,d\mu \leq c\lambda^{p-1}\,\|\nu\|.
\end{equation}

Now we turn our attention to the first integral on the right hand side of \eqref{eqdecomp94}.
We write
\begin{align*}
\int \left(\sup_{S\ni x}\frac{\wit \nu_b(B_S)}{\ell(S)^n}\right)^p\,d\mu(x)
&=
\int_{\bigcup_j 2Q_j} \ldots
+
\int_{\R^d\setminus \bigcup_j 2Q_j} \ldots =:I_1+I_2.
\end{align*}
To estimate $I_1$, we claim that
$$\frac{\wit \nu_b(B_S)}{\ell(S)^n}\lesssim\lambda.$$
This follows from the fact that $\tilde cB_S$ contains some cube $R_j$, which in turn implies that, for some $\eta\geq 6$ with
$\eta\approx \ell(S)/\ell(Q_j)$, $B_S$
is contained in some cube $\eta Q_j$ with $\ell(\eta Q_j)\approx \ell(S)$, and then
$$\frac{\wit \nu_b(B_S)}{\ell(S)^n}\lesssim \frac{\nu(\eta Q_j)}{\ell(\eta Q_j)^n},$$
which together with \eqref{eqad43} and \eqref{lema CZ 2} yields the claim above.
Then, using also \eqref{lema CZ 1} and the fact the cubes $\{Q_j\}_j$ have finite overlap, we deduce that
$$I_1 \lesssim \lambda^p\,\sum_j\mu(2Q_j)\lesssim \lambda^p\,\sum_j\frac{\nu(Q_j)}\lambda \lesssim \lambda^{p-1}\,\|\nu\|.$$

Finally we deal with the integral $I_2$.
Consider $x\in \R^d\setminus \bigcup_j 2Q_j$ and  $S$ such that $x\in S\in\CC$ (which, in particular, tells us that
$S\setminus \bigcup_j 2Q_j\neq\varnothing$). Notice that
$$\wit \nu_b(B_S) \leq \sum_{i:Q_i\cap B_S\neq\varnothing}\nu(Q_i).$$
From the conditions
$Q_i\cap B_S\neq\varnothing$ and $S\setminus \bigcup_j 2Q_j\neq\varnothing$, we infer that
$r(B_S)\geq \frac12 \ell(Q_i)$. So we deduce that $Q_i\subset c_4B_S$, for some constant $c_4\gtrsim 1$. Hence,
$$\wit \nu_b(B_S) \leq \sum_{i:Q_i\subset c_4 B_S}\nu(Q_i) \leq \sum_{i:Q_i\subset c_4 B_S} \int b_i\,d\mu,$$
%% A la darer desigualtatt usem que \nu es positiva. Si raonem am al variacio, no podem dir que \nu(Q_i) \leq \int |b_i|\,d\mu
where we used \eqref{lema CZ 4} for the last estimate. Observe now that if $Q_i\subset c_4 B_S$, then $R_i\subset c_5 B_S$, for some
absolute constant $c_5\geq c_4$. So recalling that
$\wit g = \sum_j b_j$,
we obtain
$$\wit \nu_b(B_S)\lesssim \int_{c_5B_S}\wit g\,d\mu,$$
Therefore,
$$\frac{\wit \nu_b(B_S)}{\ell(S)^n}\lesssim \frac1{\mu(B_S)}\int_{c_5B_S}\wit g\,d\mu\lesssim \wit\MM\wit g(x)$$
for every $x\in S$. So arguing as in \eqref{eqgg423} we deduce that
$$I_2\lesssim \int (\wit\MM\wit g(x))^p\,d\mu(x)\lesssim\lambda^{p-1}\,\|\nu\|.$$
Together with the estimate we obtained for $I_1$, this yields
\begin{equation}\label{eqdh23}
\int \left(\sup_{S\ni x}\frac{\wit \nu_b(B_S)}{\ell(S)^n}\right)^p\,d\mu(x)\lesssim\lambda^{p-1}\,\|\nu\|,
\end{equation}
and so using \eqref{eqgg423} we get \eqref{eqdk211}.

In order to show \eqref{eqdk212}, recall that $\nu=\wit\nu_b+f\mu$ with $f$ as in (\ref{lema CZ 3}). Thus,
\begin{equation*}
\begin{split}
\nu(B_S)=\wit\nu_b(B_S)+\int_{B_S}f\,d\mu
\lesssim\wit\nu_b(B_S)+m_{B_S}f\,\ell(S)^n,
\end{split}
\end{equation*}
and then
\begin{equation}\label{eqdk212xsxs}
\sum_{S\in\CC} \left(\frac{\nu(B_S)}{\ell(S)^n}\right)^p a_S
\lesssim\,\sum_{S\in\CC} \left(\frac{\wit\nu_b(B_S)}{\ell(S)^n}\right)^p a_S+\sum_{S\in\CC} \left(m_{B_S}f\right)^p a_S.
\end{equation}
We easily get \eqref{eqdk212} from \eqref{eqdk212xsxs}, combinig \eqref{eqcarleson1} and \eqref{eqcarleson2} in Theorem \ref{teocarleson} with \eqref{eqdh23} and the fact that $\|f\|_{L^p(\mu)}^p\leq\lambda^{p-1}\|\nu\|$ by \eqref{lema CZ 3}.
\end{proof}
\vspace{2mm}

Let $\mu$ be a uniformly $n$-rectifiable measure in $\R^d$. Consider the splitting
$\DD^\mu=\BB\cup(\biguplus_{T\in\Trees}T)$ given by the corona decomposition of $\mu$.
For a fixed constant $A\geq1$, we denote by
$\partial T$ the family of cubes $Q\in T$ for which either $Q=Q_T$ with $Q_T$ as in $(d)$ in Section \ref{5ss corona decomposition} or there exists some $P\in\DD^\mu\setminus T$ such that
\begin{equation}\label{frontera tree}
\frac12 \,\ell(P)\leq \ell(Q)\leq 2\ell(P)\quad\mbox{ and }\quad \dist(P,Q)\leq A\,\ell(Q).
\end{equation}
We call $\partial T$ the boundary of $T$. If $T=\Tree(R)$, with $R\in\Top_\GG$, we also write
$\partial\Tree(R):=\partial T$. We set
$$\partial\Trees := \bigcup_{T\in\Trees} \partial T.$$
 Notice that $\partial T\subset T$.

The following lemma has been proved in \cite[(3.28) in page 60]{David-Semmes}.

\begin{lemma}
Let $\mu$ be a uniformly $n$-rectifiable measure in $\R^d$. The family $\partial\Trees$
 is a Carleson family.
\end{lemma}

We will also need the following auxiliary result.

\begin{lemma}[Annuli estimates]\label{annulus}
Assume that the constants $\eta$ and $\theta$ in property $(f)$ of the corona decomposition (see {\em Section \ref{5ss corona decomposition}}) are small enough.
Let $Q\in\DD^\mu$, $x\in Q$ and $\epsilon\in[\ell(Q)/2,\ell(Q)]$. Let $k\in\Z$ be  such that $2^{-k}\leq\ell(Q)$. Given $R\in V(Q)$ {and $C>0$},
denote
$$\Lambda_k:=\left\{P\in \Tree(R)\cup \Stop(R):\,\ell(P)=2^{-k},\, P\subset A(x,\epsilon-C2^{-k},\epsilon+C2^{-k})\right\}.$$
Then
\begin{equation}\label{eq4974}
\mu\left(\textstyle{\bigcup_{P\in\Lambda_k}P}\right)\lesssim2^{-k}\ell(R)^{n-1},
\end{equation}
where the implicit constant in the last inequality above only depends on $n$, $d$, $\mu$ and $C$.
\end{lemma}
In the lemma, if $\epsilon-C2^{-k}<0$ we set
$A(x,\epsilon-C2^{-k},\epsilon+C2^{-k}):=\overline{B(x,\epsilon+C2^{-k})}$.
For the proof, see \cite[Lemma 5.9]{MT}. In fact, in this reference the annuli estimates are proved only for $R\in\GG$. However, for $R\in\BB$, the inequality \eqref{eq4974} is trivial. {Further, in \cite[Lemma 5.9]{MT} one states that the result holds only for some constant $C$ depending on $n$, $d$, and the AD-regularity constant of $\mu$, and with a slight difference in the definition of $V(Q)$. However, it is trivial to check that this extends to the more general version above.}

%\vspace{2mm}

\section{$\VV_\rho\circ\TT:M(\R^d)\to L^{1,\infty}(\mu)$ is a bounded operator}\label{sec3}

In this section we will prove the following result.

\begin{theorem}\label{weakL1 thm sl}
Let $\mu$ be a uniformly $n$-rectifiable measure in $\R^d$. Let $K$ be an odd kernel satisfying \eqref{4eq333} and consider the operator $T$ associated to $K$ defined in \eqref{eqtkmu**}.
 Then, for $\rho>2$,
\begin{itemize}
\item [$(i)$] $\VV^\SSS_\rho\circ\TT:M(\R^d)\to L^{1,\infty}(\mu)$  is bounded,
\item [$(ii)$] $\VV^\LL_\rho\circ\TT:M(\R^d)\to L^{1,\infty}(\mu)$  is bounded.
\end{itemize}
In particular, $\VV_\rho\circ\TT$ is a bounded operator from $M(\R^d)$ to $L^{1,\infty}(\mu)$ for all $\rho>2$.
\end{theorem}

Notice that by the triangle inequality we can easily split the variation operator into the short and long variations, that is,
$(\VV_\rho\circ\TT^\mu)f\leq (\VV^\SSS_\rho\circ\TT^\mu)f+(\VV^\LL_\rho\circ\TT^\mu)f.$
Therefore, that $\VV_\rho\circ\TT$ is a bounded operator from $M(\R^d)$ to $L^{1,\infty}(\mu)$ for all $\rho>2$ follows from  $(i)$ and $(ii)$ above, whose proofs are given below.

We will use the next result, which is contained in \cite[Theorem 1.3 and Corollary 4.2]{MT}.

\begin{theorem}\label{weakL1 thm smooth}
Let $\mu$ be a uniformly $n$-rectifiable measure in $\R^d$. Let $K$ be an odd kernel satisfying \eqref{4eq333} and consider the operator $T$ associated to $K$ defined in \eqref{eqtkmu**}.
 Then, for $\rho>2$,
\begin{itemize}
\item [$(i)$] $\VV_\rho\circ\TT^\mu:L^{2}(\mu)\to L^{2}(\mu)$  is bounded,
\item [$(ii)$] $\VV_\rho\circ\TT_\varphi:M(\R^d)\to L^{1,\infty}(\mu)$  is bounded.
\end{itemize}
\end{theorem}

\begin{proof}[Proof of {\em Theorem \ref{weakL1 thm sl}$(ii)$}]
We will deal with the long variation $\VV^\LL_\rho\circ\TT$ by comparing it with the smoothened version $\VV_\rho\circ\TT_\varphi$, using Theorem \ref{weakL1 thm smooth}$(ii)$, estimating the error terms by the short variation $\VV^\SSS_\rho\circ\TT$, and applying Theorem \ref{weakL1 thm sl}$(i)$. More precisely, the triangle inequality yields
\begin{equation*}
\begin{split}
|T_{\epsilon}\nu(x)-T_{\delta}\nu(x)|
\leq|T_{\varphi_\epsilon}\nu(x)-T_{\varphi_\delta}\nu(x)|
+|T_{\epsilon}\nu(x)-T_{\varphi_\epsilon}\nu(x)|
+|T_{\delta}\nu(x)-T_{\varphi_\delta}\nu(x)|
\end{split}
\end{equation*}
for any $0<\delta\leq\epsilon$. Therefore,
\begin{equation}\label{eq1}
\begin{split}
\big((&\VV^\LL_\rho\circ\TT)\nu(x)\big)^{\rho}
\lesssim\big((\VV_\rho\circ\TT_\varphi)\nu(x)\big)^\rho\\
&\quad+\sup_{\{\epsilon_m\}}\sum_{\begin{subarray}{c}m\in\Z:\,\epsilon_m\in I_j,\,\epsilon_{m+1}\in I_k\\ \text{ for some }j<k\end{subarray}}
\Big(|T_{\epsilon_m}\nu(x)-T_{\varphi_{\epsilon_m}}\nu(x)|^\rho
+|T_{\epsilon_{m+1}}\nu(x)-T_{\varphi_{\epsilon_{m+1}}}\nu(x)|^\rho\Big)\\
&\lesssim\big((\VV_\rho\circ\TT_\varphi)\nu(x)\big)^\rho
+\sup_{\begin{subarray}{c}\{\epsilon_m\}:\,\epsilon_m\in I_m\\
\text{for all }m\in\Z\end{subarray}}\sum_{m\in\Z}
|T_{\epsilon_m}\nu(x)-T_{\varphi_{\epsilon_m}}\nu(x)|^\rho.
\end{split}
\end{equation}

Let us estimate the second term on the right hand side of (\ref{eq1}).
Since $\chi_{[4,\infty)}\leq\varphi_\R\leq\chi_{[1/4,\infty)}$ by definition,  we have
$$\chi_{[1,\infty)}(t)-\varphi_{\R}(t)=
\int_{1/4}^4\varphi'_\R(s)(\chi_{[1,\infty)}
-\chi_{[s,\infty)})(t)\,ds$$
for all $t\geq0$.
This means that $\chi_{[1,\infty)}-\varphi_{\R}$ is a convex combination of the functions $\chi_{[1,\infty)}-\chi_{[s,\infty)}$ for $1/4\leq s\leq 4$. Then, Fubini's theorem gives
\begin{equation}\label{eq2}
\begin{split}
T_{\epsilon}\nu(x)-T_{\varphi_\epsilon}\nu(x)
&=\int\Big(\chi_{[1,\infty)}(|x-y|^2/\epsilon^2)-\varphi_\R(|x-y|^2/\epsilon^2)\Big)K(x-y)\,d\nu(y)\\
&=\int_{1/4}^4\varphi'_\R(s)\int(\chi_{[\epsilon,\infty)}-\chi_{[\epsilon\sqrt{s},\infty)})
(|x-y|)K(x-y)\,d\nu(y)\,ds\\
&=\int_{1/4}^4\varphi'_\R(s)\Big(T_\epsilon\nu(x)-T_{\epsilon\sqrt{s}}\nu(x)\Big)\,ds.
\end{split}
\end{equation}

It is easy to see that
\begin{equation}\label{eq3}
\bigg(\sum_{m\in\Z}
|T_{\epsilon_m}\nu(x)-T_{\epsilon_m\sqrt{s}}\nu(x)|^\rho\bigg)^{1/\rho}
\lesssim(\VV^\SSS_\rho\circ\TT)\nu(x)
\end{equation} for all $s\in[1/4,4]$ with uniform bounds, where $\{\epsilon_m\}_{m\in\Z}$ is any sequence such that  $\epsilon_m\in I_m$ for all $m\in\Z$. Using (\ref{eq2}), Minkowski's integral inequality and (\ref{eq3}), we get
\begin{equation}\label{eq4}
\begin{split}
\sup_{\begin{subarray}{c}\{\epsilon_m\}:\,\epsilon_m\in I_m\\
\text{for all }m\in\Z\end{subarray}}\bigg(\sum_{m\in\Z}&
|T_{\epsilon_m}\nu(x)-T_{\varphi_{\epsilon_m}}\nu(x)|^\rho\bigg)^{1/\rho}
\\
&\leq\sup_{\begin{subarray}{c}\{\epsilon_m\}:\,\epsilon_m\in I_m\\\text{for all }m\in\Z\end{subarray}}
\int_{1/4}^4\varphi'_\R(s)\bigg(\sum_{m\in\Z}
|T_{\epsilon_m}\nu(x)-T_{\epsilon_m\sqrt{s}}\nu(x)|^\rho\bigg)^{1/\rho}
\,ds\\
&\lesssim\int_{1/4}^4\varphi'_\R(s)(\VV^\SSS_\rho\circ\TT)\nu(x)\,ds
\lesssim(\VV^\SSS_\rho\circ\TT)\nu(x).
\end{split}
\end{equation}
Finally, applying (\ref{eq4}) to (\ref{eq1}) yields
\begin{equation*}
(\VV^\LL_\rho\circ\TT)\nu(x)
\lesssim(\VV_\rho\circ\TT_\varphi)\nu(x)
+(\VV^\SSS_\rho\circ\TT)\nu(x),
\end{equation*}
and Theorem \ref{weakL1 thm sl}$(ii)$ follows by
Theorems \ref{weakL1 thm smooth}$(ii)$ and  \ref{weakL1 thm sl}$(i)$.
\end{proof}

\begin{proof}[Proof of {\em Theorem \ref{weakL1 thm sl}$(i)$}]
We have to prove that there exists some constant $C>0$ such that
\begin{equation}\label{eq5}
\mu\big(\big\{x\in\R^d\,:\,(\VV^\SSS_\rho\circ\TT)\nu(x)>\lambda\big\}\big)
\leq \frac{C}{\lambda}\,\|\nu\|
\end{equation}
for all $\nu\in M(\R^d)$ and all $\lambda>0$.
The proof of (\ref{eq5}) combines the Calder\'{o}n-Zygmund decomposition developed in Lemma \ref{lema CZ}, the corona decomposition of $\mu$ described in Subsection \ref{5ss corona decomposition}, and other standard techniques for proving variational inequalities. We will start following the lines of the proof of \cite[Theorem 1.4]{M}, until the application of the
corona decomposition.

Since $\VV^\SSS_\rho\circ\TT$ is sublinear, we can assume without loss of generality that $\nu$ is a positive measure. Let us first check that we can also assume both $\mu$ and $\nu$ to be compactly supported. Given $\nu\in M(\R^d)$ and $M\in\N$, set
$$\nu_M := \chi_{B(0,2^M)}\nu.$$ If $\diam(\supp\mu)<+\infty$ then $\mu$ is compactly supported. In case $\diam(\supp\mu)=+\infty$ we are going to restrict $\mu$ to a set $K_N\subset \R^d$ such that $\mu|_{K_N}$ it is still uniformly rectifiable (with constants independent of $N$). For this purpose, for each $N\in \N$ consider the family of cubes $P_i^N\in\DD^\mu_{-N}$, $i\in I_N$, (thus $\ell(P_i^N)=2^N$ for all $i\in I_N$) such that
$B(0,2^N)\cap P_i^N\neq\varnothing$. We denote
$$K_N= \bigcup_{i\in I_N} P_i^N \quad \mbox{ and }\quad \mu_N= \mu|_{K_N}.$$
It is immediate to check that $\mu|_{P_i^N}$ is uniformly rectifiable for each $i,N$. Since $K_N$ is a finite union of uniformly rectifiable sets (because $\# I_N$ is uniformly bounded), $\mu_N$ is also
 uniformly rectifiable, with constants independent of $N$.

Suppose that there exists some constant $C>0$ such that
\begin{equation*}%\label{eq6}
\mu_{N}\big(\big\{x\in\R^d\,:\,(\VV^\SSS_\rho\circ\TT)\nu_M(x)>\lambda\big\}\big)\leq\frac{C}{\lambda}\,\|\nu_M\|
\end{equation*}
for all $\lambda>0$, all $\nu\in M(\R^d)$ and all $M,N\in\N$.
This implies that
\begin{equation}\label{eq6}
\mu\big(\big\{x\in B(0,2^N):\,(\VV^\SSS_\rho\circ\TT)\nu_M(x)>\lambda\big\}\big)\leq\frac{C}{\lambda}\,\|\nu_M\|
\end{equation}
for all $\lambda>0$, all $\nu\in M(\R^d)$ and all $M,N\in\N$.
 It is not hard to show that
$$\left|(\VV^\SSS_\rho\circ\TT)\nu(x)-(\VV^\SSS_\rho\circ\TT)\nu_N(x)\right|
\leq\frac{C'}{(2^M-2^N)^n}\,\nu\big(\R^d\setminus B(0,2^M)\big)$$
for all $x\in B(0,2^N)$ and all $M>N>1$.
In particular, if $M\to\infty$ then $(\VV^\SSS_\rho\circ\TT)\nu_M(x)\to (\VV^\SSS_\rho\circ\TT)\nu(x)$  uniformly in $B(0,2^N)$. Since (\ref{eq6}) holds for $\nu_M$ by assumption, we deduce that it also holds for $\nu$. Now, by letting $N\to\infty$ and using monotone convergence, (\ref{eq6}) with $\nu_M$ replaced by $\nu$ yields  (\ref{eq5}),
as desired. In conclusion, for proving the theorem, we only have to verify (\ref{eq5}) when $\mu$ and $\nu$ have compact support. Moreover, since (\ref{eq5}) obviously holds for $\lambda\leq2^{d+1}\|\nu\|/\|\mu\|$, we can also restrict ourselves to the case $\lambda>2^{d+1}\|\nu\|/\|\mu\|$.

We are going to verify that we can assume \eqref{eqsupp}, which will allows us to use \eqref{xcvbvcx} in the sequel, when we pursue the Calder\'on-Zygmund decomposition of $\nu$ with respect to $\mu$. Let $M:=\diam(\supp\mu)<+\infty$ and set $\nu_c:=\chi_{\R^d\setminus\mathcal{U}_M(\supp\mu)}\nu$. Then $\dist(\supp\nu_c,\supp\mu)\geq M$. By Chebyshev's inequality,
\begin{equation}\label{eq7vvvv}
\begin{split}
\mu\big(\big\{x\in\R^d\,&:\,(\VV^\SSS_\rho\circ\TT)\nu_c(x)>\lambda\big\}\big)
\leq\frac{1}{\lambda}\int(\VV^\SSS_\rho\circ\TT)\nu_c(x)\,d\mu(x)\\
&\leq\frac{C}{\lambda}\iint|x-y|^{-n}\,d\nu_c(y)\,d\mu(x)
\leq \frac{C}{M^{n}\lambda}\,\|\nu_c\|\|\mu\|.
\end{split}
\end{equation}
For any $x\in\supp\mu$, $\|\mu\|=\mu(B(x,M))\lesssim M^n$ by the AD regularity assumption on $\mu$. Thus \eqref{eq7vvvv} yields
\begin{equation}\label{eq7vvvvv}
\mu\big(\big\{x\in\R^d\,:\,(\VV^\SSS_\rho\circ\TT)\nu_c(x)>\lambda\big\}\big)\leq\frac{C}{\lambda}\,\|\nu_c\|\leq\frac{C}{\lambda}\,\|\nu\|,
\end{equation}
with $C$ independent of $M$. Note that $\nu=\nu_c+(\nu-\nu_c)$ and $\supp(\nu-\nu_c)\subset\mathcal{U}_{\diam(\supp\mu)}(\supp\mu)$. Using that $\VV^\SSS_\rho\circ\TT$ is sublinear and \eqref{eq7vvvvv} we see that, in order to prove the theorem, it is enough to show that
\begin{equation*}
\mu\big(\big\{x\in\R^d\,:\,(\VV^\SSS_\rho\circ\TT)(\nu-\nu_c)(x)>\lambda\big\}\big)
\leq \frac{C}{\lambda}\,\|\nu\|,
\end{equation*}
that is, we can assume that $\nu$ satisfies \eqref{eqsupp}. In conclusion, for proving \eqref{eq5}, from now on we assume that both $\mu$ and $\nu$ are compactly supported and they satisfy \eqref{eqsupp}.

Let $\{Q_j\}_j$ be the almost disjoint
family of cubes of Lemma \ref{lema CZ}, and set $\Omega:=\bigcup_jQ_j$ and $R_j:=6Q_j$.
Then we can write $\nu=g\mu+\nu_b$, with
$$g\mu:=\chi_{\R^d\setminus\Omega}\nu+ \sum_j b_j\mu\quad\text{and}\quad
\nu_b:=\sum_j\nu_b^j:=\sum_j\left(w_j\nu-b_j\mu\right),$$
where the $b_j$'s satisfy (\ref{lema CZ 4}), (\ref{lema CZ 5}) and (\ref{lema CZ 6}), and
$w_j:=\chi_{Q_j}\left(\sum_k \chi_{Q_k}\right)^{-1}$. Since \eqref{eqsupp} holds, in the sequel we can also assume that \eqref{xcvbvcx} holds.

Since $\VV^\SSS_\rho\circ\TT$ is sublinear,
\begin{equation}\label{eq7}
\begin{split}
&\mu\big(\big\{x\in\R^d\,:\,(\VV^\SSS_\rho\circ\TT)\nu(x)>\lambda\big\}\big)\\
&\qquad\leq\mu\big(\big\{x\in\R^d\,:\,(\VV^\SSS_\rho\circ\TT^\mu)g(x)>{\lambda}/{2}\big\}\big)
+\mu\big(\big\{x\in\R^d\,:\,(\VV^\SSS_\rho\circ\TT)\nu_b(x)>{\lambda}/{2}\big\}\big).
\end{split}
\end{equation}
We obviously have $\VV^\SSS_\rho\circ\TT^\mu\leq\VV_\rho\circ\TT^\mu$, so Theorem \ref{weakL1 thm smooth}$(i)$ yields that $\VV^\SSS_\rho\circ\TT^{\mu}$ is bounded in $L^2(\mu)$. Note that $|g|\leq C\lambda$ by (\ref{lema CZ 3}) and (\ref{lema CZ 6}). Hence, using (\ref{lema CZ 5}),
\begin{equation}\label{eq8}
\begin{split}
\mu\big(\big\{x\in\R^d:(\VV^\SSS_\rho\circ\TT^\mu)g(x)>{\lambda}/{2}\big\}\big)
&\lesssim\frac{1}{\lambda^{2}}\int|(\VV^\SSS_\rho\circ\TT^\mu)g|^2\,d\mu
\lesssim\frac{1}{\lambda^{2}}\int|g|^2\,d\mu\\
&\lesssim\frac{1}{\lambda}\int|g|\,d\mu
\leq\frac{1}{\lambda}\bigg(\nu(\R^d\setminus\Omega)+\sum_j\int_{R_j}|b_j|\,d\mu\bigg)\\
&\leq\frac{1}{\lambda}\bigg(\nu(\R^d\setminus\Omega)+\sum_j\nu(Q_j)\bigg)\lesssim\frac{\|\nu\|}{\lambda}.
\end{split}
\end{equation}

Set $\wih\Omega:=\bigcup_j2Q_j$. By (\ref{lema CZ 1}), we have
$\mu(\wih\Omega)\leq\sum_j\mu(2Q_j)\lesssim\lambda^{-1}\sum_j\nu(Q_j)\lesssim\lambda^{-1}\|\nu\|$. We are going to prove that
\begin{equation}\label{eq9}
\mu\big(\big\{x\in\R^d\setminus\wih\Omega\,:\,(\VV^\SSS_\rho\circ\TT)\nu_b(x)>{\lambda}/{2}\big\}\big)\lesssim\frac{\|\nu\|}{\lambda}.
\end{equation}
Then (\ref{eq5}) follows directly from (\ref{eq7}), (\ref{eq8}), (\ref{eq9}) and the estimate $\mu(\wih\Omega)\lesssim\lambda^{-1}\|\nu\|$ above-mentioned, finishing the proof of Theorem \ref{weakL1 thm sl}$(i)$.

To prove (\ref{eq9}), given $x\in\R^d\setminus\wih\Omega$ we first write
\begin{equation}\label{eq10}
\begin{split}
(\VV^\SSS_\rho\circ\TT)\nu_b(x)
&\leq(\VV^\SSS_\rho\circ\TT)\bigg(\sum_j\chi_{2R_j}(x)\nu_b^j\bigg)(x)
+(\VV^\SSS_\rho\circ\TT)\bigg(\sum_j\chi_{\R^d\setminus2R_j}(x)\nu_b^j\bigg)(x).
\end{split}
\end{equation}
Notice that  $\chi_{2R_j}(x)$ and $\chi_{\R^d\setminus2R_j}(x)$ are evaluated at the fixed point $x$ on the right hand side.

The first term on the right hand side of (\ref{eq10}) is easily handled using the $L^2(\mu)$ boundedness of $\VV^\SSS_\rho\circ\TT^\mu$ and standard estimates. More precisely, since $\VV^\SSS_\rho\circ\TT$ is sublinear,
\begin{equation}\label{eq11}
\begin{split}
(\VV^\SSS_\rho\circ\TT)\bigg(\sum_j&\chi_{2R_j}(x)\nu_b^j\bigg)(x)\\
&\leq\sum_j\chi_{2R_j}(x)(\VV^\SSS_\rho\circ\TT^\mu)b_j(x)
+\sum_j\chi_{2R_j}(x)(\VV^\SSS_\rho\circ\TT^\nu)w_j(x)
\end{split}
\end{equation}
because $\nu_b^j=w_j\nu-b_j\mu$. On one hand, using Theorem \ref{weakL1 thm smooth}$(i)$, that $\mu(2R_j)\lesssim\mu(R_j)$ (by \eqref{xcvbvcx}) and (\ref{lema CZ 5}), we get
\begin{equation}\label{eq12}
\begin{split}
\int_{2R_j}(\VV^\SSS_\rho\circ\TT^\mu)b_j\,d\mu
&\leq\bigg(\int_{2R_j}|(\VV^\SSS_\rho\circ\TT^\mu)b_j|^2
\,d\mu\bigg)^{1/2}\mu(2R_j)^{1/2}
\\&\lesssim\|b_j\|_{L^2(\mu)}\mu(2R_j)^{1/2}
\lesssim\|b_j\|_{L^\infty(\mu)}\mu(R_j)\lesssim\nu(Q_j).
\end{split}
\end{equation}
On the other hand, if $x\in2R_j\setminus2Q_j$ then $\dist(x,Q_j)
\approx\ell(Q_j)$. Therefore, given $k\in\Z$,
\begin{equation}\label{eq13}
\begin{split}
B(x,2^{-k})\cap Q_j=\varnothing\Longleftrightarrow
\dist(x,Q_j)\geq2^{-k}\Longleftrightarrow\ell(Q_j)\gtrsim2^{-k}.
\end{split}
\end{equation}
Since the $\ell^\rho$-norm is not bigger than the $\ell^1$-norm for $\rho\geq1$, and since $\supp w_j\subset Q_j$ and $|w_j|\leq1$, from (\ref{eq13}) and \eqref{T_de} we get
\begin{equation*}
\begin{split}
(\VV^\SSS_\rho\circ\TT^\nu)w_j(x)
&\leq\sup_{\{\epsilon_m\}}
\sum_{k\in\Z}\,\sum_{\epsilon_m,\epsilon_{m+1}\in I_k}
|T^\nu_{\epsilon_m,\epsilon_{m+1}}w_j(x)|\\
&\lesssim\nu(Q_j)
\sum_{k\in\Z:\,B(x,2^{-k})\cap Q_j\neq\varnothing}2^{kn}
\lesssim\nu(Q_j)\ell(Q_j)^{-n},
\end{split}
\end{equation*}
and therefore, using again that $\mu(2R_j)\lesssim \mu(R_j)\approx\ell(R_j)^n\approx\ell(Q_j)^{n}$ by \eqref{xcvbvcx}, we obtain
\begin{equation}\label{eq14}
\begin{split}
\int_{2R_j\setminus2Q_j}(\VV^\SSS_\rho\circ\TT^\nu)w_j\,d\mu
&\lesssim\nu(Q_j)\ell(Q_j)^{-n}\mu(2R_j)
\lesssim\nu(Q_j).
\end{split}
\end{equation}
Finally, applying (\ref{eq12}) and (\ref{eq14}) to (\ref{eq11}), we conclude that
\begin{equation}\label{eq15}
\begin{split}
\int_{\R^d\setminus\wih\Omega}&(\VV^\SSS_\rho\circ\TT)\bigg(\sum_j\chi_{2R_j}(x)\nu_b^j\bigg)(x)\,d\mu(x)\\
&\leq\sum_j\int_{2R_j}(\VV^\SSS_\rho\circ\TT^\mu)b_j\,d\mu
+\sum_j\int_{2R_j\setminus2Q_j}(\VV^\SSS_\rho\circ\TT^\nu)w_j\,d\mu
\lesssim\sum_j\nu(Q_j)\lesssim\|\nu\|.
\end{split}
\end{equation}
Thanks to (\ref{eq10}), (\ref{eq15}) and Chebyshev's inequality, to prove (\ref{eq9}) it is enough to verify that
\begin{equation}\label{eq16}
\mu\bigg(\bigg\{x\in\R^d\setminus\wih\Omega\,:\,(\VV^\SSS_\rho\circ\TT)\bigg(\sum_j\chi_{\R^d\setminus2R_j}(x)\nu_b^j\bigg)(x)>{\lambda}/{4}\bigg\}\bigg)\lesssim\frac{\|\nu\|}{\lambda}.
\end{equation}

Our task now is to prove (\ref{eq16}). Given $x\in\supp\mu$, let $\{\epsilon_m\}_{m\in\Z}$ be a non-increasing sequence of positive numbers (which depends on $x$, i.e. $\epsilon_m\equiv\epsilon_m(x)$) such that
\begin{equation}\label{eq17}
(\VV^\SSS_\rho\circ\TT)\bigg(\!\sum_j\chi_{\R^d\setminus2R_j}(x)\nu_b^j\bigg)(x)
\leq2\bigg(\sum_{k\in\Z}\sum_{\epsilon_m,\epsilon_{m+1}\in I_k}
\!\bigg|\!\sum_j\!
\chi_{\R^d\setminus2R_j}(x)T_{\epsilon_m,\epsilon_{m+1}}\nu_b^j(x)\bigg|^\rho\bigg)^{1/\rho}\!.
\end{equation}
Typically, the problem of the existence of such a sequence can be avoided by defining an auxiliary operator $\VV^\SSS_{\rho,I}\circ\TT$ along the same lines of $\VV^\SSS_\rho\circ\TT$ and requiring the supremum to be taken over a finite set of indices $I$ (thus the supremum is a maximum in this case). One then proves the desired estimate for $\VV^\SSS_{\rho,I}\circ\TT$ with bounds independent of $I$ and deduces the general result by taking the supremum over all finite sets $I$ and using monotone convergence. For the sake of shortness, we omit the details.

Define the {\em interior} and {\em boundary sum}, respectively, by
\begin{equation*}
\begin{split}
S_i(x)&:=\bigg(\sum_{k\in\Z}\,\sum_{\epsilon_m,\epsilon_{m+1}\in I_k}\bigg|\sum_{j:\,R_j\subset A(x,\epsilon_{m+1},\epsilon_m)}
\chi_{\R^d\setminus2R_j}(x)
T_{\epsilon_m,\epsilon_{m+1}}\nu_b^j(x)\bigg|^\rho\bigg)^{1/\rho},\\
S_b(x)&:=\bigg(\sum_{k\in\Z}\,\sum_{\epsilon_m,\epsilon_{m+1}\in I_k}\bigg|\sum_{j:\,R_j\cap\partial A(x,\epsilon_{m+1},\epsilon_m)\neq\varnothing}\chi_{\R^d\setminus2R_j}(x)
T_{\epsilon_m,\epsilon_{m+1}}\nu_b^j(x)\bigg|^\rho\bigg)^{1/\rho}.
\end{split}
\end{equation*}
If $R_j\cap A(x,\epsilon_{m+1},\epsilon_m)=\varnothing$ then $T_{\epsilon_m,\epsilon_{m+1}}\nu_b^j(x)=0$, thus
$$(\VV^\SSS_\rho\circ\TT)\bigg(\!\sum_j\chi_{\R^d\setminus2R_j}(x)\nu_b^j\bigg)(x)\leq2(S_i+S_b)$$ by (\ref{eq17}) and the triangle inequality, and so
\begin{equation}\label{eq18}
\begin{split}
\mu\bigg(\bigg\{x\in&\R^d\setminus\wih\Omega\,:\,(\VV^\SSS_\rho\circ\TT)\bigg(\sum_j\chi_{\R^d\setminus2R_j}(x)\nu_b^j\bigg)(x)>{\lambda}/{4}\bigg\}\bigg)\\
&\leq\mu\big(\big\{x\in\R^d\setminus\wih\Omega\,:\,S_i(x)>\lambda/16\big\}\big)
+\mu\big(\big\{x\in\R^d\setminus\wih\Omega\,:\,S_b(x)>\lambda/16\big\}\big).
\end{split}
\end{equation}

To estimate $\mu\big(\big\{x\in\R^d\setminus\wih\Omega\,:\,S_i(x)>\lambda/16\big\}\big)$ we use the fact that the $\ell^\rho$-norm is not bigger than the $\ell^1$-norm for $\rho\geq1$, and that $\supp(\nu_b^j)\subset R_j$:
\begin{equation}\label{eq19}
\begin{split}
S_i(x)&\leq\sum_{m\in\Z}\bigg|\sum_{j:\,R_j\subset A(x,\epsilon_{m+1},\epsilon_m)}\chi_{\R^d\setminus2R_j}(x)
T_{\epsilon_m,\epsilon_{m+1}}\nu_b^j(x)\bigg|\\
&\leq\sum_{j}\chi_{\R^d\setminus2R_j}(x)\sum_{m\in\Z:\,A(x,\epsilon_{m+1},\epsilon_m)\supset R_j}
\!\!|T_{\epsilon_m,\epsilon_{m+1}}\nu_b^j(x)|\leq\sum_{j}\chi_{\R^d\setminus 2R_j}(x)|T\nu_b^j(x)|,
\end{split}
\end{equation}
Recall that $\nu_b^j(R_j)=0$ and $\|\nu_b^j\|\lesssim\nu(Q_j)$ by (\ref{lema CZ 5}). Thus, if $z_j$ denotes the center of $R_j$, we have
\begin{equation}\label{eq20}
\begin{split}
\int_{\R^d\setminus 2R_j}|T\nu_b^j|\,d\mu
&\leq\int_{\R^d\setminus 2R_j}\int_{R_j}|K(x-y)-K(x-z_j)|\,d|\nu_b^j|(y)\,d\mu(x)\\
&\lesssim\int_{\R^d\setminus 2R_j}\int_{R_j}\frac{|y-z_j|}{|x-z_j|^{n+1}}\,d|\nu_b^j|(y)\,d\mu(x)\\
&\lesssim\|\nu_b^j\|\int_{\R^d\setminus 2R_j}\frac{\ell(R_j)}{|x-z_j|^{n+1}}\,d\mu(x)\lesssim\|\nu_b^j\|
\lesssim\nu(Q_j).
\end{split}
\end{equation}
Finally, from Chebyshev's inequality, (\ref{eq19}) and (\ref{eq20}) we conclude that
\begin{equation}\label{eq21}
\begin{split}
\mu\big(\big\{x\in\R^d\setminus\wih\Omega\,:\,S_i(x)>\lambda/16\big\}\big)
&\leq\frac{16}{\lambda}
\sum_j\int_{\R^d\setminus2R_j}|T\nu_b^j|\,d\mu
\lesssim\frac{1}{\lambda}\sum_j\nu(Q_j)\lesssim\frac{\|\nu\|}{\lambda}.
\end{split}
\end{equation}

By (\ref{eq18}), (\ref{eq21}) and Chebyshev's inequality once again we see that, in order to prove (\ref{eq16}), it is enough to show that
\begin{equation}\label{eq22}
\begin{split}
\int_{\R^d\setminus\wih\Omega}S_b^2\,d\mu
&\lesssim\lambda\|\nu\|.
\end{split}
\end{equation}
The proof of this estimate is much more involved than the previous ones and requires the use of the corona decomposition of
$\mu$, that is, we need to introduce the splitting $\DD^\mu=\BB\cup(\biguplus_{S\in\Trees}S)$. We denote
$$T_{j,m}(x):=\chi_{\R^d\setminus2R_j}(x)
T_{\epsilon_m,\epsilon_{m+1}}\nu_b^j(x).$$
Recall that for $P\in\DD_k$ we write $I_P=[2^{-k-1},2^{-k})$.
Since $\rho>2$, the $\ell^\rho$-norm is not bigger than the $\ell^2$-norm, and we get
\begin{equation}\label{eq23}
\begin{split}
\int_{\R^d\setminus\wih\Omega}S_b^2\,d\mu
&\leq\sum_{P\in\BB}\int_{P\setminus\wih\Omega}\sum_{\epsilon_m,\epsilon_{m+1}\in I_P}
\bigg|\sum_{j:\,R_j\cap\partial A(x,\epsilon_{m+1},\epsilon_m)\neq\varnothing}T_{j,m}(x)\bigg|^2\,d\mu(x)\\
&\quad+\sum_{S\in\Trees}\sum_{P\in S}\int_{P\setminus\wih\Omega}\sum_{\epsilon_m,\epsilon_{m+1}\in I_P}
\bigg|\sum_{j:\,R_j\cap\partial A(x,\epsilon_{m+1},\epsilon_m)\neq\varnothing}T_{j,m}(x)\bigg|^2\,d\mu(x).
\end{split}
\end{equation}
Observe that
\begin{equation}\label{eq26}
|T_{j,m}(x)|\lesssim\ell(P)^{-n}\chi_{\R^d\setminus2R_j}(x)|\nu_b^j|(A(x,\epsilon_{m+1},\epsilon_m))
\end{equation}
for all $\epsilon_m,\epsilon_{m+1}\in I_P$. If in addition $x\in P\setminus2R_j$ and $R_j\cap \partial A(x,\epsilon_{m+1},\epsilon_m)\neq\varnothing$, taking into account that $\epsilon_m\approx \epsilon_{m+1}\approx \ell(P)\gtrsim \dist(x,R_j)\gtrsim \ell(R_j)$,
we deduce that
\begin{equation}\label{eq42**}
R_j\subset B_P,
\end{equation}
assuming the constant $c_1$ in \eqref{defbq} big enough.

Concerning the first term on the right hand side of (\ref{eq23}), from (\ref{eq26}) and using that $\|\nu_b^j\|\lesssim\nu(Q_j)$, that the $Q_j$'s have bounded overlap and that $Q_j\subset B_P$ for all $j$ such that $R_j\subset B_P$, we get
\begin{equation}\label{eq24}
\begin{split}
\sum_{P\in\BB}\int_{P\setminus\wih\Omega}&\sum_{\epsilon_m,\epsilon_{m+1}\in I_P}
\bigg|\sum_{j:\,R_j\cap\partial A(x,\epsilon_{m+1},\epsilon_m)\neq\varnothing}T_{j,m}(x)\bigg|^2\,d\mu(x)\\
&\lesssim\sum_{P\in\BB}\int_{P}
\bigg(\sum_{\epsilon_m,\epsilon_{m+1}\in I_P}\sum_{j:\,R_j\subset B_P}
\ell(P)^{-n}|\nu_b^j|(A(x,\epsilon_{m+1},\epsilon_m))\bigg)^2\,d\mu(x)\\
&\lesssim\sum_{P\in\BB}\int_{P}
\bigg(\sum_{j:\,R_j\subset B_P}
\frac{\|\nu_b^j\|}{\ell(P)^{n}}\bigg)^2\,d\mu(x)
\lesssim\sum_{P\in\BB:\exists R_j\subset B_P}
\bigg(\frac{\nu(B_P)}{\ell(P)^{n}}\bigg)^2\ell(P)^n
\lesssim\lambda\|\nu\|,
\end{split}
\end{equation}
where we also used Lemma \ref{lemcarleson} in the last inequality, because $\BB$ is a Carleson family.

From now on, all our efforts are devoted to estimate the second term on the right hand side of (\ref{eq23}).

\begin{claim}\label{claim1}
Assume $c_1$ in \eqref{defbq} is big enough, and let also $\alpha>0$ be big enough depending on $n$, $d$, and on the AD regualrity constants of $\mu$. Given $Q\in\Top_\GG$, $P\in\Tree(Q)$ and $R_j\subset B_P$, at least one of the following holds:
\begin{itemize}
\item[$(i)$] There exists $R\in\Tree(Q)$ such that $R\subset \alpha B_P$, $R_j\subset B_R$ and $\ell(R_j)\in I_R$.
\item[$(ii)$] There exists $R\in\partial\Tree(Q)$ such that $R\subset \alpha B_P$ and $R_j\subset B_R$.
\end{itemize}
\end{claim}

We postpone the proof of the preceding statement till the end of the proof of the theorem.
Thanks to this claim, given $Q\in\Top_\GG$ and $P\in\Tree(Q)$ we can split
$$\{j:\,R_j\subset B_P\}\subset J_1\cup J_2,$$
where
\begin{equation*}
\begin{split}
J_1:&=\{j:\,R_j\subset B_P,\, \exists\, R\in\Tree(Q) \text{ such that } R\subset \alpha B_P,\, R_j\subset B_R,\,\ell(R_j)\in I_R\},\\
J_2:&=\{j:\,R_j\subset B_P,\, \exists\, R\in\partial\Tree(Q) \text{ such that } R\subset \alpha B_P,\, R_j\subset B_R\}.
\end{split}
\end{equation*}

Recall that if $x\in P\setminus2R_j$, $\epsilon_{m},\epsilon_{m+1}\in I_P$ and $R_j\cap\partial A(x,\epsilon_{m+1},\epsilon_m)\neq\varnothing$ then $R_j\subset B_P$ (see \eqref{eq42**}). Thus, we can decompose the second term on the right hand side of (\ref{eq23}) using $J_1$ and $J_2$ as follows
\begin{equation}\label{eq28}
\begin{split}
\sum_{S\in\Trees}&\sum_{P\in S}
\int_{P\setminus\wih\Omega}\sum_{\epsilon_m,\epsilon_{m+1}\in I_P}
\bigg|\sum_{j:\,R_j\cap\partial A(x,\epsilon_{m+1},\epsilon_m)\neq\varnothing}T_{j,m}(x)\bigg|^2\,d\mu(x)\\
&\lesssim\sum_{Q\in\Top_\GG}\sum_{P\in \Tree(Q)}\int_{P\setminus\wih\Omega}\sum_{\epsilon_m,\epsilon_{m+1}\in I_P}
\bigg|\sum_{j\in J_1:\,R_j\cap\partial A(x,\epsilon_{m+1},\epsilon_m)\neq\varnothing}T_{j,m}(x)\bigg|^2\,d\mu(x)\\
&\quad+\sum_{Q\in\Top_\GG}\sum_{P\in \Tree(Q)}\int_{P\setminus\wih\Omega}\sum_{\epsilon_m,\epsilon_{m+1}\in I_P}
\bigg|\sum_{j\in J_2:\,R_j\cap\partial A(x,\epsilon_{m+1},\epsilon_m)\neq\varnothing}T_{j,m}(x)\bigg|^2\,d\mu(x).
\end{split}
\end{equation}
Despite that the arguments to estimate both terms on the right hand side of (\ref{eq28}) are similar, we will deal with them separately, due to its different nature with respect to the structure of the corona decomposition.

\begin{claim}\label{claim2}
Let $Q$, $P$, $x$, $\epsilon_m$ and $\epsilon_{m+1}$ be as on the right hand side of
 $\eqref{eq28}$. We have
\begin{equation}\label{eq25}
\begin{split}
\bigg|&\sum_{\begin{subarray}{c}j\in J_1:\, R_j\cap\partial A(x,\epsilon_{m+1},\epsilon_m)\neq\varnothing\end{subarray}}\!\!
|\nu_b^j|(A(x,\epsilon_{m+1},\epsilon_m))\bigg|^2\\
&\quad\qquad\lesssim\lambda\ell(P)^n\sum_{k:\,2^{-k}\lesssim\ell(P)}
\bigg(\frac{2^{-k}}{\ell(P)}\bigg)^{1/2}
\sum_{\begin{subarray}{c}j\in J_1:\,\ell(R_j)\in I_k\end{subarray}}
|\nu_b^j|(A(x,\epsilon_{m+1},\epsilon_m)).
\end{split}
\end{equation}
Given $j\in J_2$, denote by $R(j)\in\partial\Tree(Q)$ some cube such that $R(j)\subset \alpha B_P$ and $R_j\subset B_{R(j)}$, where $\alpha>0$ is as in {\em Claim \ref{claim1}}. We have\begin{equation}\label{eq30}
\begin{split}
&\bigg|\sum_{\begin{subarray}{c}j\in J_2:\, R_j\cap\partial A(x,\epsilon_{m+1},\epsilon_m)\neq\varnothing\end{subarray}}\!\!
|\nu_b^j|(A(x,\epsilon_{m+1},\epsilon_m))\bigg|^2\\
&\qquad\lesssim\lambda^{1/2}\ell(P)^{n/2}\nu(B_P)^{1/2}
\sum_{\begin{subarray}{c}R\in\partial\Tree(Q):\\R\subset \alpha B_P\end{subarray}}
\sum_{\begin{subarray}{c}j\in J_2:\\ R(j)=R\end{subarray}}\!\!\bigg(\frac{\ell(R)}{\ell(P)}\bigg)^{1/4}\!\!
|\nu_b^j|(B_R\cap A(x,\epsilon_{m+1},\epsilon_m)).
\end{split}
\end{equation}
\end{claim}

Again we postpone the proof of the preceding claim till the end of the proof of the theorem.

For the case $j\in J_1$  in (\ref{eq28}), using (\ref{eq26}), (\ref{eq25}) and \eqref{eq42**} we get
\begin{equation}\label{eq27}
\begin{split}
\sum_{Q\in\Top_\GG}&\sum_{P\in \Tree(Q)}
\int_{P\setminus\wih\Omega}\sum_{\epsilon_m,\epsilon_{m+1}\in I_P}
\bigg|\sum_{j\in J_1:\,R_j\cap\partial A(x,\epsilon_{m+1},\epsilon_m)\neq\varnothing}T_{j,m}(x)\bigg|^2\,d\mu(x)\\
&\!\!\!\lesssim\lambda\sum_{Q\in\Top_\GG}\sum_{P\in \Tree(Q)}\ell(P)^{-n}\\
&\!\!\!\quad\quad\times\int_{P\setminus\wih\Omega}\sum_{\epsilon_m,\epsilon_{m+1}\in I_P}\sum_{k:\,2^{-k}\lesssim\ell(P)}
\bigg(\frac{2^{-k}}{\ell(P)}\bigg)^{1/2}\!\!\!\!\!
\sum_{\begin{subarray}{c}j\in J_1:\,\ell(R_j)\in I_k\end{subarray}}\!\!\!
|\nu_b^j|(A(x,\epsilon_{m+1},\epsilon_m))\,d\mu(x)\\
&\!\!\!\lesssim\lambda\sum_{Q\in\Top_\GG}\sum_{P\in\Tree(Q)}
\sum_{k:\,2^{-k}\lesssim\ell(P)}
\bigg(\frac{2^{-k}}{\ell(P)}\bigg)^{1/2}
\sum_{\begin{subarray}{c}j\in J_1:\,\ell(R_j)\in I_k\end{subarray}}
\|\nu_b^j\|\\
&\!\!\!\lesssim\lambda
\sum_{j}\nu(Q_j)
\sum_{k:\,\ell(R_j)\in I_k}
\sum_{\begin{subarray}{c}P\in\DD^\mu:\,R_j\subset B_P\\
2^{-k}\lesssim\ell(P)\end{subarray}}
\bigg(\frac{2^{-k}}{\ell(P)}\bigg)^{1/2}
\lesssim\lambda
\sum_{j}\nu(Q_j)
\lesssim\lambda\|\nu\|.
\end{split}
\end{equation}
In the third inequality we used that $j\in J_1$ implies that $R_j\subset B_P$.

Concerning the case $j\in J_2$ in (\ref{eq28}), by (\ref{eq26}) and (\ref{eq30}) we see that
\begin{equation*}
\begin{split}
\sum_{Q\in\Top_\GG}&\sum_{P\in \Tree(Q)}\int_{P\setminus\wih\Omega}\sum_{\epsilon_m,\epsilon_{m+1}\in I_P}
\bigg|\sum_{j\in J_2:\,R_j\cap\partial A(x,\epsilon_{m+1},\epsilon_m)\neq\varnothing}T_{j,m}(x)\bigg|^2\,d\mu(x)\\
&\!\!\!\lesssim\lambda^{1/2}\sum_{Q\in\Top_\GG}\sum_{P\in \Tree(Q)}
\ell(P)^{-n}\bigg(\frac{\nu(B_P)}{\ell(P)^n}\bigg)^{1/2}\\
&\!\!\!\quad\quad\times\int_{P\setminus\wih\Omega}\sum_{\epsilon_m,\epsilon_{m+1}\in I_P}\sum_{\begin{subarray}{c}R\in\partial\Tree(Q):\\R\subset \alpha B_P\end{subarray}}
\sum_{\begin{subarray}{c}j\in J_2:\\ R(j)=R\end{subarray}}\!\!\bigg(\frac{\ell(R)}{\ell(P)}\bigg)^{1/4}\!\!
|\nu_b^j|(B_R\cap A(x,\epsilon_{m+1},\epsilon_m))\,d\mu(x)\\
&\!\!\!\lesssim\lambda^{1/2}\sum_{Q\in\Top_\GG}\sum_{P\in \Tree(Q)}
\bigg(\frac{\nu(B_P)}{\ell(P)^n}\bigg)^{1/2}
\sum_{\begin{subarray}{c}R\in\partial\Tree(Q):\\R\subset \alpha B_P\end{subarray}}
\sum_{\begin{subarray}{c}j\in J_2:\\ R(j)=R\end{subarray}}\!\!\bigg(\frac{\ell(R)}{\ell(P)}\bigg)^{1/4}\!\!
\|\nu_b^j\|\\
&\!\!\!\lesssim\lambda^{1/2}\sum_{Q\in\Top_\GG}\sum_{P\in \Tree(Q)}
\sum_{\substack{R\in\partial\Tree(Q):\,R\subset \alpha B_P\\ \exists R_j\subset B_R}}
\bigg(\frac{\ell(R)}{\ell(P)}\bigg)^{1/4}
\bigg(\frac{\nu(B_P)}{\ell(P)^n}\bigg)^{1/2}
\bigg(\frac{\nu(B_R)}{\ell(R)^n}\bigg)\ell(R)^n,
\end{split}
\end{equation*}
where we also used in the last inequality above that $\|\nu_b^j\|\lesssim\nu(Q_j)$ and that the $Q_j$'s have bounded overlap.
Since $a^{1/2}b\lesssim a^{3/2}+b^{3/2}$ for all $a,b\geq0$, we obtain
\begin{equation*}
\begin{split}
&\sum_{Q\in\Top_\GG}\sum_{P\in \Tree(Q)}\int_{P\setminus\wih\Omega}\sum_{\epsilon_m,\epsilon_{m+1}\in I_P}
\bigg|\sum_{j\in J_2:\,R_j\cap\partial A(x,\epsilon_{m+1},\epsilon_m)\neq\varnothing}T_{j,m}(x)\bigg|^2\,d\mu(x)\\
&\qquad\lesssim\lambda^{1/2}\sum_{Q\in\Top_\GG}\sum_{P\in \Tree(Q)}
\sum_{\substack{R\in\partial\Tree(Q):\,R\subset \alpha B_P\\\exists R_j\subset B_R}}\!\!\!
\bigg(\bigg(\frac{\nu(B_P)}{\ell(P)^n}\bigg)^{3/2}\!\!+
\bigg(\frac{\nu(B_R)}{\ell(R)^n}\bigg)^{3/2}\bigg)
\bigg(\frac{\ell(R)}{\ell(P)}\bigg)^{1/4}\ell(R)^n\\
&\qquad\lesssim\lambda^{1/2}\sum_{Q\in\Top_\GG}\sum_{\substack{P\in \Tree(Q)\\\exists R_j\subset cB_P}}
\bigg(\frac{\nu(B_P)}{\ell(P)^n}\bigg)^{3/2}a_P
+\lambda^{1/2}\sum_{Q\in\Top_\GG}
\sum_{\substack{R\in\partial\Tree(Q)\\\exists R_j\subset B_R}}
\bigg(\frac{\nu(B_R)}{\ell(R)^n}\bigg)^{3/2}\ell(R)^n,
\end{split}
\end{equation*}
where we have set $a_P:=\sum_{R\in\partial\Tree(Q):\,R\subset \alpha B_P}
(\ell(R)/\ell(P))^{1/4}\ell(R)^n$ whenever $P\in\Tree(Q)$ for some $Q\in\Top_\GG$ (otherwise, we set $a_P=0$). Since $\partial\Trees$ is a Carleson family, we see that the $a_P$'s satisfy a Carleson packing condition because, for a given $T\in\DD^\mu$,
\begin{equation*}
\begin{split}
\sum_{P\subset T}a_P
&\leq\sum_{P\subset T}\sum_{Q\in\Top_\GG:\,P\in\Tree(Q)}
\sum_{R\in\partial\Tree(Q):\,R\subset \alpha B_P}
\bigg(\frac{\ell(R)}{\ell(P)}\bigg)^{1/4}\ell(R)^n\\
&\leq\sum_{P\subset T}
\sum_{R\in\partial\Trees:\,R\subset \alpha B_P\subset \alpha B_T}
\bigg(\frac{\ell(R)}{\ell(P)}\bigg)^{1/4}\ell(R)^n\\
&\leq\sum_{R\in\partial\Trees:\,R\subset \alpha B_T}
\ell(R)^n\sum_{P\subset T:\,R\subset \alpha B_P}
\bigg(\frac{\ell(R)}{\ell(P)}\bigg)^{1/4}
\lesssim\sum_{R\in\partial\Trees:\,R\subset \alpha B_T}
\ell(R)^n\lesssim\ell(T)^n.
\end{split}
\end{equation*}
Therefore,
\begin{equation}\label{eq29}
\begin{split}
\sum_{Q\in\Top_\GG}&\sum_{P\in \Tree(Q)}\int_{P\setminus\wih\Omega}\sum_{\epsilon_m,\epsilon_{m+1}\in I_P}
\bigg|\sum_{j\in J_2:\,R_j\cap\partial A(x,\epsilon_{m+1},\epsilon_m)\neq\varnothing}T_{j,m}(x)\bigg|^2\,d\mu(x)\\
&\lesssim\lambda^{1/2}\sum_{P\in\DD^\mu}
\bigg(\frac{\nu(B_P)}{\ell(P)^n}\bigg)^{3/2}
\big(a_P+\ell(P)^n\chi_{\partial\Trees}(P)\big)
\lesssim\lambda\|\nu\|,
\end{split}
\end{equation}
because the coefficients $a_P+\ell(P)^n\chi_{\partial\Trees}(P)$ satisfy a Carleson packing condition and thus we can use Lemma \ref{lemcarleson}.

Finally, (\ref{eq22}) follows from (\ref{eq23}), (\ref{eq24}), (\ref{eq28}), (\ref{eq27}) and (\ref{eq29}), so Theorem \ref{weakL1 thm sl}$(i)$ is proved except for the claims.
\end{proof}

\begin{proof}[Proof of {\em Claim \ref{claim1}}]
Let $Q\in\Top_\GG$, $P\in\Tree(Q)$ and $R_j\subset B_P$. For the purpose of the claim, we can assume that $\ell(Q)\geq\ell(R_j)$, otherwise we can take $R=Q$ which fulfills $(ii)$. Without loss of generality, we can also assume that $\ell(P)\geq\ell(R_j)$ (recall that $R_j\subset B_P$, so $\ell(P)\gtrsim\ell(R_j)$). Otherwise, we replace $P$ by a suitable ancestor from $\Tree(Q)$ with side length comparable to $\ell(R_j)$, which must exists thanks to the previous assumption $\ell(Q)\geq\ell(R_j)$.

Let $R\in\Tree(Q)$ be a cube with minimal side length such that $R_j\subset B_R$ and $\ell(R)\geq\ell(R_j)$, that is, $\ell(R)\leq\ell(S)$ for all $S\in\Tree(Q)$ with  $R_j\subset B_S$ and $\ell(S)\geq\ell(R_j)$. In particular, notice that $P$ may coincide with $R$, and in any case $\ell(R)\leq\ell(P)$. If $\ell(R_j)\in I_R$, that is $\ell(R)\geq\ell(R_j)\geq\ell(R)/2$, then $R$ fulfills $(i)$ if $\alpha$ is big enough, and we are done. On the contrary, assume that $\ell(R)/2>\ell(R_j)$. Since $R_j\subset B_R$ and $R_j\cap\supp\mu\neq\varnothing$, there exists $R'\in\DD^\mu$ such that $\ell(R')=\ell(R)$, $\dist(R',R)\lesssim\ell(R)$ and
$R'\cap R_j\neq\varnothing$. Therefore, there exists a son $R''$ of $R'$ such that $R''\cap R_j\neq\varnothing$, so $R_j\subset B_{R''}$ if $c_1$ is big enough. By the minimality of $R$, we must have $R''\notin\Tree(Q)$, thus $R\in\partial\Tree(Q)$ if $A\geq1$ in \eqref{frontera tree} is big enough, and then $(ii)$ is fulfilled for some $\alpha$ big enough.
\end{proof}

\begin{proof}[Proof of {\em Claim \ref{claim2}}]
Let us first prove (\ref{eq25}). If $j\in J_1$ then $R_j\subset B_P$ and, in particular, $\ell(R_j)\lesssim\ell(P)$. Thus, by Cauchy-Schwarz inequality,
\begin{equation}\label{claim2eq1}
\begin{split}
\bigg|&\sum_{\begin{subarray}{c}j\in J_1:\, R_j\cap\partial A(x,\epsilon_{m+1},\epsilon_m)\neq\varnothing\end{subarray}}\!\!
|\nu_b^j|(A(x,\epsilon_{m+1},\epsilon_m))\bigg|^2\\
&\qquad=\bigg|\sum_{k:\,2^{-k}\lesssim\ell(P)}
\bigg(\frac{2^{-k}}{\ell(P)}\bigg)^{1/4}
\bigg(\frac{\ell(P)}{2^{-k}}\bigg)^{1/4}\!\!
\sum_{\begin{subarray}{c}j\in J_1:\,\ell(R_j)\in I_k\\ R_j\cap\partial A(x,\epsilon_{m+1},\epsilon_m)\neq\varnothing\end{subarray}}
|\nu_b^j|(A(x,\epsilon_{m+1},\epsilon_m))\bigg|^2\\
&\qquad\lesssim\sum_{k:\,2^{-k}\lesssim\ell(P)}
\bigg(\frac{\ell(P)}{2^{-k}}\bigg)^{1/2}\,\,
\bigg|\sum_{\begin{subarray}{c}j\in J_1:\,\ell(R_j)\in I_k\\ R_j\cap\partial A(x,\epsilon_{m+1},\epsilon_m)\neq\varnothing\end{subarray}}
|\nu_b^j|(A(x,\epsilon_{m+1},\epsilon_m))\bigg|^2.
\end{split}
\end{equation}
Using that $|\nu_b^j|(A(x,\epsilon_{m+1},\epsilon_m))\lesssim\nu(Q_j)$ and that the $Q_j$'s have bounded overlap, from the definition of $J_1$ we see that
\begin{equation}\label{claim2eq2}
\begin{split}
\sum_{\begin{subarray}{c}j\in J_1:\,\ell(R_j)\in I_k\\ R_j\cap\partial A(x,\epsilon_{m+1},\epsilon_m)\neq\varnothing\end{subarray}}\!\!
|\nu_b^j|(A(x,\epsilon_{m+1},\epsilon_m))
&\lesssim\sum_{\begin{subarray}{c}R\in\Tree(Q):\,
\ell(R)\in I_k,\\ B_R\cap\partial A(x,\epsilon_{m+1},\epsilon_m)\neq\varnothing,\\ R\subset\alpha B_P,\,\exists R_j\subset B_R\end{subarray}}\!\!
\nu(B_R).
\end{split}
\end{equation}
If $6Q_j=R_j\subset B_R$ then $\nu(6Q_j)\leq\nu(B_R)\lesssim \lambda\mu(B_R)\lesssim\lambda\mu(R)$ by (\ref{lema CZ 2}). From
\eqref{claim2eq2} we infer
\begin{equation}\label{claim2eq2xx}
\begin{split}
\sum_{\begin{subarray}{c}j\in J_1:\,\ell(R_j)\in I_k\\ R_j\cap\partial A(x,\epsilon_{m+1},\epsilon_m)\neq\varnothing\end{subarray}}\!\!
|\nu_b^j|(A(x,\epsilon_{m+1},\epsilon_m))
&\lesssim\lambda
\sum_{\begin{subarray}{c}R\in\Tree(Q):\,
\ell(R)\in I_k,\\ B_R\cap\partial A(x,\epsilon_{m+1},\epsilon_m)\neq\varnothing,\\ R\subset\alpha B_P,\,\exists R_j\subset B_R\end{subarray}}\!\!
\mu(R).%\lesssim \lambda 2^{-k}\ell(P)^{n-1}.
\end{split}
\end{equation}
We want to show that the right hand side of \eqref{claim2eq2xx} can be estimated by $\lambda 2^{-k}\ell(P)^{n-1}$. To this end, we can suppose that $\ell(R)\leq\ell(P)$, otherwise the estimate becomes trivial because we are already assuming $2^{-k}\lesssim\ell(P)$ and $\ell(R)\in I_k$ (so in this last case there is only a finite and uniformly bounded number of terms in the sum above). Suppose now that $\ell(R)\leq\ell(P)$. Since $R\subset\alpha B_P$ then $R\subset\bigcup_{P'\in V(P)}P'$ if the constant $C_1$ in the definition of $V(P)$ is big enough. Thus, $R\subset P'$ for some $P'\in V(P)$. Note that $P'\in\Tree(Q)$ because $R\in\Tree(Q)$, and so we finally get $R\in\Tree(P')$. Then, from \eqref{claim2eq2xx} and the estimates on annuli from Lemma \ref{annulus} we obtain
\begin{equation}\label{claim2eq2xxx}
\begin{split}
\sum_{\begin{subarray}{c}j\in J_1:\,\ell(R_j)\in I_k\\ R_j\cap\partial A(x,\epsilon_{m+1},\epsilon_m)\neq\varnothing\end{subarray}}\!\!
|\nu_b^j|(A(x,\epsilon_{m+1},\epsilon_m))
&\lesssim\lambda
\sum_{P'\in V(P)}\sum_{\begin{subarray}{c}R\in\Tree(P'):\,
\ell(R)\in I_k,\\ B_R\cap\partial A(x,\epsilon_{m+1},\epsilon_m)\neq\varnothing\end{subarray}}\!\!
\mu(R)\\
&\lesssim \lambda 2^{-k}\ell(P)^{n-1},
\end{split}
\end{equation}
as desired. Finally, (\ref{eq25}) follows from (\ref{claim2eq1}) and (\ref{claim2eq2xxx}).

Let us turn our attention to (\ref{eq30}) now. Recall that, given $j\in J_2$,  $R(j)\in\partial\Tree(Q)$ denotes some cube such that $R(j)\subset \alpha B_P$ and $R_j\subset B_{R(j)}$. Similarly to (\ref{claim2eq1}), by H\"older's inequality we get
\begin{equation}\label{claim2eq3}
\begin{split}
&\bigg|\sum_{\begin{subarray}{c}j\in J_2:\\ R_j\cap\partial A(x,\epsilon_{m+1},\epsilon_m)\neq\varnothing\end{subarray}}\!\!
|\nu_b^j|(A(x,\epsilon_{m+1},\epsilon_m))\bigg|^{3/2}\\
&\quad\leq\bigg|\sum_{\begin{subarray}{c}R\in\partial\Tree(Q):\,R\subset \alpha B_P\\ B_R\cap\partial A(x,\epsilon_{m+1},\epsilon_m)\neq\varnothing\end{subarray}}
\sum_{j\in J_2:\, R(j)=R}
|\nu_b^j|(B_R\cap A(x,\epsilon_{m+1},\epsilon_m))\bigg|^{3/2}\\
&\quad\lesssim\sum_{k:\,2^{-k}\lesssim\ell(P)}
\bigg(\frac{\ell(P)}{2^{-k}}\bigg)^{1/4}
\bigg|\!\!\!\sum_{\begin{subarray}{c}R\in\partial\Tree(Q):\\ R\subset \alpha B_P,\,\ell(R)=2^{-k}\\ B_R\cap\partial A(x,\epsilon_{m+1},\epsilon_m)\neq\varnothing\end{subarray}}
\!\!\!\sum_{j\in J_2:\, R(j)=R}
|\nu_b^j|(B_R\cap A(x,\epsilon_{m+1},\epsilon_m))\bigg|^{3/2}.
\end{split}
\end{equation}
For the cubes $R=R(j)$ in the last sum above, note that $R_j\subset B_{R}$ (see the definition of $J_2$). So, as we did before \eqref{claim2eq2xx},  $\nu(B_R)\lesssim \lambda\mu(B_{R})\lesssim\lambda\mu(R)$ by (\ref{lema CZ 2}).
Using that $\|\nu_b^j\|\lesssim\nu(Q_j)$, that the $Q_j$'s have bounded overlap and that $\nu(B_{R})\lesssim\lambda\mu(B_{R})$, we deduce that
\begin{equation}\label{claim2eq3x}
\begin{split}
&\sum_{\begin{subarray}{c}R\in\partial\Tree(Q):\, R\subset \alpha B_P,\,\ell(R)=2^{-k}\\ B_R\cap\partial A(x,\epsilon_{m+1},\epsilon_m)\neq\varnothing\end{subarray}}\,
\sum_{j\in J_2:\, R(j)=R}
|\nu_b^j|(B_R\cap A(x,\epsilon_{m+1},\epsilon_m))\\
&\qquad\qquad\qquad\lesssim\!\!\!\!
\sum_{\begin{subarray}{c}R\in\partial\Tree(Q):\\ R\subset \alpha B_P,\,\ell(R)=2^{-k}\\ B_R\cap\partial A(x,\epsilon_{m+1},\epsilon_m)\neq\varnothing\end{subarray}}\!\!
\sum_{j\in J_2:\, R(j)=R}\nu(Q_j)
\lesssim\!\!\!\!
\sum_{\begin{subarray}{c}R\in\partial\Tree(Q):\\ R\subset \alpha B_P,\,\ell(R)=2^{-k}\\ B_R\cap\partial A(x,\epsilon_{m+1},\epsilon_m)\neq\varnothing\end{subarray}}\!\!\!\nu(B_R)\\
&\qquad\qquad\qquad\lesssim\lambda\sum_{\begin{subarray}{c}R\in\partial\Tree(Q):\, R\subset \alpha B_P,\,\ell(R)=2^{-k}\\ B_R\cap\partial A(x,\ \epsilon_{m+1},\epsilon_m)\neq\varnothing\end{subarray}}\!\!\!\mu(R).
\end{split}
\end{equation}
As we did in the case of $J_1$, now we want to show that the last term above can be estimated by $\lambda2^{-k}\ell(P)^{n-1}$. We argue similarly to what we did before \eqref{claim2eq2xxx}. If $R$ is as in the right hand side of the last inequality in \eqref{claim2eq3x}, since $R\subset \alpha B_P$ we have $\ell(R)\lesssim\ell(P)$, and thus we can assume $\ell(R)\leq\ell(P)$ (otherwise the estimate that we want to show becomes trivial). Since $R\subset\alpha B_P$ then $R\subset\bigcup_{P'\in V(P)}P'$ if the constant $C_1$ in the definition of $V(P)$ is big enough. Thus, $R\subset P'$ for some $P'\in V(P)$ and $R\in\Tree(P')$ (recall that $R\in\partial\Tree(Q)$ implies $R\in\Tree(Q)$). Then, from \eqref{claim2eq3x} and the estimates on annuli from Lemma \ref{annulus} we obtain
\begin{equation}\label{claim2eq3xx}
\begin{split}
&\sum_{\begin{subarray}{c}R\in\partial\Tree(Q):\, R\subset \alpha B_P,\,\ell(R)=2^{-k}\\ B_R\cap\partial A(x,\epsilon_{m+1},\epsilon_m)\neq\varnothing\end{subarray}}\,
\sum_{j\in J_2:\, R(j)=R}
|\nu_b^j|(B_R\cap A(x,\epsilon_{m+1},\epsilon_m))\\
&\qquad\qquad\qquad\lesssim\lambda
\sum_{P'\in V(P)}\sum_{\begin{subarray}{c}R\in\Tree(P'):\,
\ell(R)\in I_k,\\ B_R\cap\partial A(x,\epsilon_{m+1},\epsilon_m)\neq\varnothing\end{subarray}}\!\!
\mu(R)
\lesssim\lambda2^{-k}\ell(P)^{n-1},
\end{split}
\end{equation}
as desired.

Combining (\ref{claim2eq3xx}) with (\ref{claim2eq3}) we get
\begin{equation}\label{claim2eq4}
\begin{split}
&\bigg|\sum_{\begin{subarray}{c}j\in J_2:\\ R_j\cap\partial A(x,\epsilon_{m+1},\epsilon_m)\neq\varnothing\end{subarray}}\!\!
|\nu_b^j|(A(x,\epsilon_{m+1},\epsilon_m))\bigg|^{3/2}\\
&\quad\lesssim\lambda^{1/2}\ell(P)^{n/2}\!\!\sum_{k:\,2^{-k}\lesssim\ell(P)}
\!\!\bigg(\frac{2^{-k}}{\ell(P)}\bigg)^{1/4}\!\!\!\!\!
\sum_{\begin{subarray}{c}R\in\partial\Tree(Q):\\ R\subset \alpha B_P,\,\ell(R)=2^{-k}\\ B_R\cap\partial A(x,\epsilon_{m+1},\epsilon_m)\neq\varnothing\end{subarray}}
\!\!\!\sum_{\begin{subarray}{c}j\in J_2:\\ R(j)=R\end{subarray}}
|\nu_b^j|(B_R\cap A(x,\epsilon_{m+1},\epsilon_m))\\
&\quad\lesssim\lambda^{1/2}\ell(P)^{n/2}
\sum_{\begin{subarray}{c}R\in\partial\Tree(Q):\,R\subset \alpha B_P\end{subarray}}\,\sum_{j\in J_2:\, R(j)=R}\!\!\bigg(\frac{\ell(R)}{\ell(P)}\bigg)^{1/4}
|\nu_b^j|(B_R\cap A(x,\epsilon_{m+1},\epsilon_m)).
\end{split}
\end{equation}
Finally, (\ref{eq30}) is a consequence of (\ref{claim2eq4}) and the trivial estimate
\begin{equation*}
\begin{split}
&\sum_{\begin{subarray}{c}j\in J_2:\, R_j\cap\partial A(x,\epsilon_{m+1},\epsilon_m)\neq\varnothing\end{subarray}}\!\!
|\nu_b^j|(A(x,\epsilon_{m+1},\epsilon_m))
\lesssim\nu(B_P),
\end{split}
\end{equation*}
which holds if $c_1$ in \eqref{defbq} is big enough because $\|\nu_b^j\|\lesssim\nu(Q_j)$ and the $Q_j$'s have bounded overlap.

\end{proof}

%\newpage
\section{$\VV_\rho\circ\TT^\mu:L^p(\mu)\to L^p(\mu)$ is a bounded operator for $1<p<\infty$}\label{sec4}

Under the assumptions of Theorem \ref{teopri},
the boundedness of $\VV_\rho\circ\TT^\mu$ in $L^p(\mu)$ for $1<p<2$ follows by interpolation, taking into account that it is bounded in $L^2(\mu)$ and from $L^1(\mu)$ to $L^{1,\infty}(\mu)$, by
Theorem \ref{weakL1 thm smooth} and
Theorem \ref{weakL1 thm sl}. So it only remains to prove the
boundedness in $L^p(\mu)$ for $2<p<\infty$. This task is carried out in the next theorem.

\begin{theorem}\label{Lp thm}
Let $\mu$ be a uniformly $n$-rectifiable measure in $\R^d$. Let $K$ be an odd kernel satisfying \eqref{4eq333} and consider the operator $T$ associated to $K$ defined in \eqref{eqtkmu**}.
 Then $\VV_\rho\circ\TT^\mu$ is a bounded operator in $L^p(\mu)$ for all $\rho>2$ and all $2<p<\infty$.
\end{theorem}

\begin{proof}
We are going to prove that if $\mu$ is a uniformly $n$-rectifiable measure then $\MM^\sharp_{\DD^\mu}\circ\VV_\rho\circ\TT^\mu$ is a bounded operator in $L^p(\mu)$ for all $2<p<\infty$, where $\MM^\sharp_{\DD^\mu}$ denotes the dyadic sharp maximal function, that is,
\begin{equation*}
\MM^\sharp_{\DD^\mu}f(x)
=\sup_{\begin{subarray}{c}D\in\DD^\mu:\,x\in D\end{subarray}}
m_D|f-m_Df|.
\end{equation*}
The theorem will then follow from the fact that the maximal operator defined by
$\MM_{\DD^\mu}f(x) =\sup_{\begin{subarray}{c}D\in\DD^\mu:\,x\in D\end{subarray}}
m_D|f|$
 can be controlled in $L^p(\mu)$ norm
by $\MM_{\DD^\mu}^\sharp$. That is, $\|\MM_{\DD^\mu}f\|_{L^p(\mu)} \lesssim \|\MM_{\DD^\mu}^\sharp f\|_{L^p(\mu)} $ (see \cite[ Lemma 6.9]{Duo}, for example).

Fix $f\in L^p(\mu)$ and $x_0\in\supp\mu$. Then,
\begin{equation}\label{Msharp 1}
(\MM^\sharp_{\DD^\mu}\circ\VV_\rho\circ\TT^\mu)f(x_0)
=\sup_{\begin{subarray}{c}D\in\DD^\mu:\,x_0\in D\end{subarray}}
m_D|(\VV_\rho\circ\TT^\mu)f-m_D((\VV_\rho\circ\TT^\mu) f)|.
\end{equation}
Given $D\in\DD^\mu$ such that $x_0\in D$, we decompose $f=f_1+f_2$ with $f_1:=f\chi_{3D}$ and $f_2:=f-f_1$. Since
$\VV_\rho\circ\TT^\mu$ is sublinear and positive,
$|(\VV_\rho\circ\TT^\mu)f-(\VV_\rho\circ\TT^\mu)f_2|\leq(\VV_\rho\circ\TT^\mu)f_1$ and so
$|(\VV_\rho\circ\TT^\mu)f-c|\leq(\VV_\rho\circ\TT^\mu)f_1+|(\VV_\rho\circ\TT^\mu)f_2-c|$ for all $c\in\R$. If we take
$c=(\VV_\rho\circ\TT^\mu) f_2(z_D)$, where $z_D$ denotes the center of $D$ (we may assume that $c<\infty$), then
\begin{equation}\label{Msharp 2}
\begin{split}
m_D|(\VV_\rho\circ\TT^\mu)f-&m_D((\VV_\rho\circ\TT^\mu) f)|\\
&\leq2m_D|(\VV_\rho\circ\TT^\mu)f-(\VV_\rho\circ\TT^\mu) f_2(z_D)|\\
&\lesssim m_D(\VV_\rho\circ\TT^\mu)f_1
+m_D|(\VV_\rho\circ\TT^\mu)f_2-(\VV_\rho\circ\TT^\mu) f_2(z_D)|\\
&=:I_1+I_2.
\end{split}
\end{equation}

A good estimate for $I_1$ can be easily derived using Cauchy-Schwarz's inequality, Theorem \ref{weakL1 thm smooth}$(i)$ and that $\mu$ is $n$-AD regular. More precisely,
\begin{equation}\label{Msharp 3}
\begin{split}
I_1\lesssim\left(\frac{1}{\mu(D)}\int_D|(\VV_\rho\circ\TT^\mu)f_1|^2\,d\mu\right)^{1/2}
\lesssim\left(\frac{1}{\mu(D)}\int_{3D}|f|^2\,d\mu\right)^{1/2}
\lesssim\MM_2f(x_0).
\end{split}
\end{equation}

The estimate of $I_2$ is much more involved. Given $x\in D$, by the triangle inequality we have
\begin{equation}\label{Msharp 4}
\begin{split}
|(\VV_\rho\circ\TT^\mu)f_2(x)&-(\VV_\rho\circ\TT^\mu)f_2(z_D)|\\
&\leq\sup_{\{\epsilon_m\}_{m\in\Z}}\bigg(\sum_{m\in\Z}
\left|T^\mu_{\epsilon_m,\epsilon_{m+1}}f_2(x)
-T^\mu_{\epsilon_m,\epsilon_{m+1}}f_2(z_D)\right|^\rho\bigg)^{1/\rho},
\end{split}
\end{equation}
where the supremum is taken over all non-increasing sequences $\{\epsilon_m\}_{m\in\Z}$ of positive numbers $\epsilon_m$.
In order to estimate the right hand side of (\ref{Msharp 4}), take one of such sequences $\{\epsilon_m\}_{m\in\Z}$ and note that, by the triangle inequality again,
\begin{equation}\label{Msharp 5}
\begin{split}
\left|T^\mu_{\epsilon_m,\epsilon_{m+1}}\right.&\left.f_2(x)
-T^\mu_{\epsilon_m,\epsilon_{m+1}}f_2(z_D)\right|\\
&\leq\int\chi_{(\epsilon_{m+1},\epsilon_m]}(|x-y|)
\left|K(x-y)-K(z_D-y)\right||f_2(y)|\,d\mu(y)\\
&\quad+\int\left|\chi_{(\epsilon_{m+1},\epsilon_m]}(|x-y|)
-\chi_{(\epsilon_{m+1},\epsilon_m]}(|z_D-y|)\right||K(z_D-y)|
|f_2(y)|\,d\mu(y)\\
&=:a_m+b_m.
\end{split}
\end{equation}
Since $x$ and $z_D$ belong to $D$ and $f_2$ vanishes in $3D$, we can assume that $\epsilon_{m+1}>\ell(D)$ in the definition of $a_m$ and $b_m$ for all $m\in\Z$.

Let us first look at the sum relative to the $a_m$'s for $m\in\Z$. Using that $\rho>1$, the regularity of the kernel $K$, that $f_2$ vanishes in $3D$, and that $\mu$ is $n$-AD regular, for each $x\in D$ we have
\begin{equation}\label{Msharp 6}
\begin{split}
\bigg(\sum_{m\in\Z}a_m^\rho\bigg)^{1/\rho}
&\leq\sum_{m\in\Z}\int_{\epsilon_{m+1}<|x-y|\leq\epsilon_m}
|K(x-y)-K(z_D-y)||f_2(y)|\,d\mu(y)\\
&\lesssim\ell(D)\sum_{m\in\Z}\int_{\epsilon_{m+1}<|x-y|\leq\epsilon_m}
\frac{|f_2(y)|}{|y-z_D|^{n+1}}\,d\mu(y)\\
&\leq\ell(D)\int_{\R^d\setminus3D}
\frac{|f(y)|}{|y-z_D|^{n+1}}\,d\mu(y)\lesssim\MM f(x_0)
\leq\MM_2 f(x_0),
\end{split}
\end{equation}
where we also used Cauchy-Schwarz's inequality in the last estimate above.

The sum relative to the $b_m$'s for $m\in\Z$ requires a more delicate analysis. We split $\Z=J_1\cup J_2$, where
\begin{equation*}
\begin{split}
J_1:=\{m\in\Z: \epsilon_m-\epsilon_{m+1}>\ell(D)\},\\
J_2:=\{m\in\Z: \epsilon_m-\epsilon_{m+1}\leq\ell(D)\}.
\end{split}
\end{equation*}
To shorten notation, we also set
\begin{equation*}
\begin{split}
A_m^1(z_D):=A(z_D,\epsilon_m-\ell(D),\epsilon_m+\ell(D))
\quad\text{and}\quad A_m^2(x):=A(x,\epsilon_{m+1},\epsilon_m).
\end{split}
\end{equation*}
Since we are assuming $\epsilon_{m+1}>\ell(D)$ for all $m\in\Z$, both  $A_m^1(z_D)$ and $A_{m+1}^1(z_D)$ are well defined for all $m\in J_1$. Moreover, since $|x-z_D|\leq\ell(D)$ for all $x\in D$, we easily get
\begin{equation}\label{Msharp 7}
\begin{split}
&\left|\chi_{(\epsilon_{m+1},\epsilon_m]}(|x-\cdot|)
-\chi_{(\epsilon_{m+1},\epsilon_m]}(|z_D-\cdot|)\right|
\leq\chi_{A_m^1(z_D)}+\chi_{A_{m+1}^1(z_D)}
\quad\text{for all } m\in J_1,\\
&\left|\chi_{(\epsilon_{m+1},\epsilon_m]}(|x-\cdot|)
-\chi_{(\epsilon_{m+1},\epsilon_m]}(|z_D-\cdot|)\right|
\leq\chi_{A_{m}^2(z_D)}+\chi_{A_m^2(x)}
\quad\text{for all } m\in J_2.
\end{split}
\end{equation}
We are going to split the sum associated with the $b_m$'s in terms of $J_1$ and $J_2$, using in each case the corresponding estimate from (\ref{Msharp 7}).

Concerning the sum over $J_1$, since $\rho>2$, (\ref{Msharp 7}) yields
\begin{equation}\label{Msharp 8}
\begin{split}
\bigg(\sum_{m\in J_1}b_m^\rho\bigg)^{1/\rho}
&\lesssim\Bigg(\sum_{m\in J_1}\bigg(
\int_{A^1_m(z_D)}|K(z_D-y)||f_2(y)|\,d\mu(y)\bigg)^2\Bigg)^{1/2}\\
&\quad+\Bigg(\sum_{m\in J_1}\bigg(
\int_{A^1_{m+1}(z_D)}|K(z_D-y)||f_2(y)|\,d\mu(y)\bigg)^2\Bigg)^{1/2}\\
&=:S_1+S_2.
\end{split}
\end{equation}
The arguments for estimating $S_1$ and $S_2$ are almost the same, so we will only give the details for $S_1$. Since $f_2$ vanishes in $3D$,
\begin{equation}\label{Msharp 9}
\begin{split}
S_1^2&=\sum_{k\in\Z}\sum_{\begin{subarray}{c}m\in J_1:\\ \epsilon_m\in I_k\end{subarray}}\bigg(
\int_{A^1_m(z_D)}|K(z_D-y)||f_2(y)|\,d\mu(y)\bigg)^2
\lesssim\!\!\!\sum_{\begin{subarray}{c}Q\in\DD^\mu:\\Q\supset D\end{subarray}}
\sum_{\begin{subarray}{c}m\in J_1:\\ \epsilon_m\in I_Q\end{subarray}}
\!\!\!\frac{\left|(|f_2|\mu)\left(A^1_m(z_D)\right)\right|^2}{\ell(Q)^{2n}}.
\end{split}
\end{equation}

Our task now is to bound $\left|(|f_2|\mu)\left(A^1_m(z_D)\right)\right|^2$. This is done by splitting the annulus $A^1_m(z_D)$, whose width equals $2\ell(D)$, into disjoint cubes $P\in\DD^\mu$ such that $\ell(P)=\ell(D)$ and grouping them properly in terms of the corona decomposition, in order to be able to apply Carleson's embedding theorem later. More precisely, for $Q\supset D$ and $\epsilon_m\in I_Q$, we have
$$A^1_m(z_D)\cap \supp(\mu)\subset\bigcup_{R\in V(Q)}R \subset
\Biggl(\,\bigcup_{R\in V(Q)}\bigcup_{\begin{subarray}{c}P\in\Tree(R):\\ \ell(P)=\ell(D)\end{subarray}} P\Biggr) \cup \Biggl(\,
\bigcup_{R\in V(Q)}\bigcup_{\begin{subarray}{c}P\in\Stop(R):\\ \ell(P)\geq\ell(D)\end{subarray}} P \Biggr)
.$$
 Recall also that the number of cubes in $V(Q)$ is bounded independently of $Q$. Therefore,
\begin{equation}\label{Msharp 10}
\begin{split}
\left|(|f_2|\mu)\left(A^1_m(z_D)\right)\right|^2
&\lesssim\sum_{R\in V(Q)}\bigg|\sum_{\begin{subarray}{c}P\in\Tree(R):\\ \ell(P)=\ell(D)\end{subarray}}(|f_2|\mu)\left(A^1_m(z_D)\cap P\right)\bigg|^2\\
&\quad+\sum_{R\in V(Q)}\bigg|\sum_{\begin{subarray}{c}P\in\Stop(R):\\ \ell(P)\geq\ell(D)\end{subarray}}(|f_2|\mu)\left(A^1_m(z_D)\cap P\right)\bigg|^2.
\end{split}
\end{equation}
The first term on the right hand side of (\ref{Msharp 10}) can be easily estimated using Cauchy-Schwarz's inequality, that the $P$'s such that $\ell(P)=\ell(D)$ are disjoint and Lemma \ref{annulus}. That is,
\begin{equation}\label{Msharp 11}
\begin{split}
&\bigg|\sum_{\begin{subarray}{c}P\in\Tree(R):\\ \ell(P)=\ell(D)\end{subarray}}(|f_2|\mu)\left(A^1_m(z_D)\cap P\right)\bigg|^2
=\bigg|\int\bigg(\sum_{\begin{subarray}{c}P\in\Tree(R):\\ \ell(P)=\ell(D)\end{subarray}}\chi_{A^1_m(z_D)\cap P}\bigg)|f_2|\,d\mu\bigg|^2\\
&\quad\qquad\qquad\leq\bigg(\sum_{\begin{subarray}{c}P\in\Tree(R):\\ \ell(P)=\ell(D)\end{subarray}}\mu\left(A^1_m(z_D)\cap P\right)\bigg)\bigg(\sum_{\begin{subarray}{c}P\in\Tree(R):\\ \ell(P)=\ell(D)\end{subarray}}(|f_2|^2\mu)\left(A^1_m(z_D)\cap P\right)\bigg)\\
&\quad\qquad\qquad\lesssim\ell(D)\ell(R)^{n-1}\sum_{\begin{subarray}{c}P\in\Tree(R):\\ \ell(P)=\ell(D)\end{subarray}}
\left(|f_2|^2\mu\right)\left(A^1_m(z_D)\cap P\right).
\end{split}
\end{equation}
The second term on the right hand side of (\ref{Msharp 10}) is estimated similarly but, since the cubes in $\Stop(R)$ may have different side length, we need to introduce an auxiliary splitting of the sum in terms of the side length. This extra splitting, combined with an application of Cauchy-Schwarz inequality yields
\begin{equation}\label{Msharp 12}
\begin{split}
&\bigg|\sum_{\begin{subarray}{c}P\in\Stop(R):\\ \ell(P)\geq\ell(D)\end{subarray}}
(|f_2|\mu)\left(A^1_m(z_D)\cap P\right)\bigg|^2
\!\!=\bigg|\sum_{j\geq0}\frac{2^{j/4}}{2^{j/4}}
\sum_{\begin{subarray}{c}P\in\Stop(R):\, \ell(P)\geq\ell(D)\\ \ell(P)=2^{-j}\ell(R)\end{subarray}}\!\!
(|f_2|\mu)\left(A^1_m(z_D)\cap P\right)\bigg|^2\\
&\quad\lesssim\sum_{j\geq0}2^{j/2}\bigg|\sum_{\begin{subarray}{c}P\in\Stop(R):\, \ell(P)\geq\ell(D)\\ \ell(P)=2^{-j}\ell(R)\end{subarray}}(|f_2|\mu)\left(A^1_m(z_D)\cap P\right)\bigg|^2\\
&\quad\leq\sum_{j\geq0}2^{j/2}\bigg(\sum_{\begin{subarray}{c}P\in\Stop(R):\\ \ell(P)\geq\ell(D)\\ \ell(P)=2^{-j}\ell(R)\end{subarray}}\mu\left(A^1_m(z_D)\cap P\right)\bigg)
\bigg(\sum_{\begin{subarray}{c}P\in\Stop(R):\\ \ell(P)\geq\ell(D)\\ \ell(P)=2^{-j}\ell(R)\end{subarray}}\left(|f_2|^2\mu\right)\left(A^1_m(z_D)\cap P\right)\bigg),
\end{split}
\end{equation}
where we also used in the last inequality above that the $P$'s which belong to $\Stop(R)$ are disjoint and Cauchy-Schwarz's inequality. Since the width of the annulus $A^1_m(z_D)$ equals $2\ell(D)$, if $P\in\Stop(R)$ is such that $\ell(P)=2^{-j}\ell(R)\geq\ell(D)$ and $A^1_m(z_D)\cap P\neq\varnothing$ then $$P\subset A(z_D,\epsilon_m-C2^{-j}\ell(R),\epsilon_m+C2^{-j}\ell(R))$$ for some $C>0$ depending only on $n$, $d$ and $\mu$. Hence, Lemma \ref{annulus} gives
\begin{equation*}
\begin{split}
\sum_{\begin{subarray}{c}P\in\Stop(R):\, \ell(P)\geq\ell(D)\\ \ell(P)=2^{-j}\ell(R)\end{subarray}}\mu\left(A^1_m(z_D)\cap P\right)\lesssim2^{-j}\ell(R)^n,
\end{split}
\end{equation*}
which plugged into (\ref{Msharp 12}) yields
\begin{equation}\label{Msharp 13}
\begin{split}
\bigg|\sum_{\begin{subarray}{c}P\in\Stop(R):\\ \ell(P)\geq\ell(D)\end{subarray}}
(|f_2|\mu)\left(A^1_m(z_D)\cap P\right)\bigg|^2
&\lesssim\sum_{j\geq0}2^{-j/2}\ell(R)^n
\!\!\sum_{\begin{subarray}{c}P\in\Stop(R):\\ \ell(P)\geq\ell(D)\\ \ell(P)=2^{-j}\ell(R)\end{subarray}}\left(|f_2|^2\mu\right)\left(A^1_m(z_D)\cap P\right)\\
&\leq\!\!\sum_{\begin{subarray}{c}P\in\Stop(R)\end{subarray}}
\left(\frac{\ell(P)}{\ell(R)}\right)^{1/2}\ell(R)^n\left(|f_2|^2\mu\right)\left(A^1_m(z_D)\cap P\right).
\end{split}
\end{equation}
Applying (\ref{Msharp 11}) and (\ref{Msharp 13}) to (\ref{Msharp 10}), we  see that
\begin{equation}\label{Msharp 14}
\begin{split}
\left|(|f_2|\mu)\left(A^1_m(z_D)\right)\right|^2
&\lesssim\sum_{R\in V(Q)}\sum_{\begin{subarray}{c}P\in\Tree(R):\\ \ell(P)=\ell(D)\end{subarray}}\frac{\ell(D)}{\ell(R)}\,\ell(R)^{n}
\left(|f_2|^2\mu\right)\left(A^1_m(z_D)\cap P\right)\\
&\quad+\sum_{R\in V(Q)}\sum_{\begin{subarray}{c}P\in\Stop(R)\end{subarray}}
\left(\frac{\ell(P)}{\ell(R)}\right)^{1/2}\ell(R)^n\left(|f_2|^2\mu\right)\left(A^1_m(z_D)\cap P\right).
\end{split}
\end{equation}

Now that we have estimated $\left|(|f_2|\mu)\left(A^1_m(z_D)\right)\right|^2$, we can derive a bound for $S_1^2$. Since $\ell(Q)=\ell(R)$ for all $R\in V(Q)$, (\ref{Msharp 9}) and (\ref{Msharp 14}) imply that
\begin{equation}\label{Msharp 15}
\begin{split}
S_1^2&\lesssim\sum_{\begin{subarray}{c}Q\in\DD^\mu:\\Q\supset D\end{subarray}}
\sum_{\begin{subarray}{c}m\in J_1:\\ \epsilon_m\in I_Q\end{subarray}}
\sum_{R\in V(Q)}\sum_{\begin{subarray}{c}P\in\Tree(R):\\ \ell(P)=\ell(D)\end{subarray}}\frac{\ell(D)}{\ell(R)}\,\ell(R)^{-n}
\left(|f_2|^2\mu\right)\left(A^1_m(z_D)\cap P\right)\\
&\quad+\sum_{\begin{subarray}{c}Q\in\DD^\mu:\\Q\supset D\end{subarray}}
\sum_{\begin{subarray}{c}m\in J_1:\\ \epsilon_m\in I_Q\end{subarray}}
\sum_{R\in V(Q)}\sum_{\begin{subarray}{c}P\in\Stop(R)\end{subarray}}
\left(\frac{\ell(P)}{\ell(R)}\right)^{1/2}\ell(R)^{-n}\left(|f_2|^2\mu\right)\left(A^1_m(z_D)\cap P\right).
\end{split}
\end{equation}
Note that,  for $m\in J_1$, each (closed) annulus $A^1_m(z_D)$ overlaps only with the two neighbors $A^1_{m-1}(z_D)$, $A^1_{m+1}(z_D)$
at the boundaries because $\{\epsilon_m\}_{m\in\Z}$ is a non-increasing sequence. Therefore, from (\ref{Msharp 15}) we deduce that
\begin{equation}\label{Msharp 16}
\begin{split}
S_1^2&\lesssim\sum_{\begin{subarray}{c}Q\in\DD^\mu:\\Q\supset D\end{subarray}}
\sum_{R\in V(Q)}\sum_{\begin{subarray}{c}P\in\Tree(R):\\ \ell(P)=\ell(D)\end{subarray}}\frac{\ell(D)}{\ell(R)}\,\ell(R)^{-n}
\left(|f_2|^2\mu\right)\left(P\right)\\
&\quad+\sum_{\begin{subarray}{c}Q\in\DD^\mu:\\Q\supset D\end{subarray}}
\sum_{R\in V(Q)}\sum_{\begin{subarray}{c}P\in\Stop(R)\end{subarray}}
\left(\frac{\ell(P)}{\ell(R)}\right)^{1/2}\ell(R)^{-n}\left(|f_2|^2\mu\right)\left(P\right).
\end{split}
\end{equation}
For the first term on the right hand side of (\ref{Msharp 16}), using that the $P$'s in $\DD^\mu$ such that $\ell(P)=\ell(D)$ are disjoint, that $\mu$ is $n$-AD regular and that $x_0\in D$, we have
\begin{equation}\label{Msharp 17}
\begin{split}
\sum_{\begin{subarray}{c}Q\in\DD^\mu:\\Q\supset D\end{subarray}}
\sum_{R\in V(Q)}\sum_{\begin{subarray}{c}P\in\Tree(R):\\ \ell(P)=\ell(D)\end{subarray}}\frac{\ell(D)}{\ell(R)}\,\ell(R)^{-n}\left(|f_2|^2\mu\right)(P)
&\leq\sum_{\begin{subarray}{c}Q\in\DD^\mu:\\Q\supset D\end{subarray}}
\sum_{R\in V(Q)}\frac{\ell(D)}{\ell(R)}
\frac{\left(|f_2|^2\mu\right)(R)}{\ell(R)^n}\\
&\lesssim\sum_{\begin{subarray}{c}Q\in\DD^\mu:\\Q\supset D\end{subarray}}
\sum_{R\in V(Q)}\frac{\ell(D)}{\ell(Q)}\,
\MM_2 f(x_0)^2\\
&\lesssim \MM_2 f(x_0)^2.
\end{split}
\end{equation}
In order to estimate the second term on the right hand side of (\ref{Msharp 16}), note that $R\in V(Q)$ if and only if $Q\in V(R)$ and that if $D\subset Q$ and $R\in V(Q)$ then $D\subset 3R$, thus
by changing the order of summation and using that the number of cubes in $V(R)$ is bounded independently of $R$ and that $\DD^\mu=\bigcup_{S\in\Top}\Tree(S)$ we see that
\begin{equation}\label{Msharp 18}
\begin{split}
\sum_{\begin{subarray}{c}Q\in\DD^\mu:\\Q\supset D\end{subarray}}
\sum_{R\in V(Q)}&\sum_{\begin{subarray}{c}P\in\Stop(R)\end{subarray}}
\left(\frac{\ell(P)}{\ell(R)}\right)^{1/2}\ell(R)^{-n}\left(|f_2|^2\mu\right)\left(P\right)\\
&\leq\sum_{\begin{subarray}{c}R\in\DD^\mu:\\3R\supset D\end{subarray}}
\sum_{Q\in V(R)}\sum_{\begin{subarray}{c}P\in\Stop(R)\end{subarray}}
\left(\frac{\ell(P)}{\ell(R)}\right)^{1/2}\ell(R)^{-n}\left(|f_2|^2\mu\right)(P)\\
&\lesssim\sum_{S\in\Top}\sum_{\begin{subarray}{c}R\in\Tree(S):\\3R\supset D\end{subarray}}
\sum_{\begin{subarray}{c}P\in\Stop(R)\end{subarray}}
\left(\frac{\ell(P)}{\ell(R)}\right)^{1/2}\ell(R)^{-n}\left(|f_2|^2\mu\right)(P)\\
&\lesssim\sum_{\begin{subarray}{c}S\in\Top\end{subarray}}\sum_{\begin{subarray}{c}P\in\Stop(S)\end{subarray}}\ell(P)^{1/2}\left(|f_2|^2\mu\right)(P)
\sum_{\begin{subarray}{c}R\in\Tree(S):\\3R\supset D\cup P\end{subarray}}
\ell(R)^{-n-1/2},
\end{split}
\end{equation}
where we also used in the last inequality above that, for $S\in\Top$, if $P\in\Stop(R)$ for some $R\in\Tree(S)$ then $P\in\Stop(S)$ and $P\subset R$. Moreover, denoting
$$D(P,D):=\ell(P)+\dist(P,D) +\ell(D),$$
we have
\begin{equation}\label{Msharp 19}
\begin{split}
\sum_{\begin{subarray}{c}R\in\Tree(S):\\3R\supset D\cup P\end{subarray}}
\ell(R)^{-n-1/2}
&\lesssim\sum_{j\in\Z}
\sum_{\begin{subarray}{c}R\in\Tree(S):\,3R\supset D\cup P,\\
2^jD(P,D)<\ell(R)\leq2^{j+1}D(P,D)\end{subarray}}(2^jD(P,D))^{-n-1/2}\\
&\lesssim D(P,D)^{-n-1/2},
\end{split}
\end{equation}
because the number of cubes $R\in\DD^\mu$ such that $3R\supset D\cup P$ and $2^jD(P,D)<\ell(R)\leq2^{j+1}D(P,D)$ is bounded independently of $j\in\Z$, and the statements ``$3R\supset D\cup P$'' and
``$2^jD(P,D)<\ell(R)\leq2^{j+1}D(P,D)$'' are compatible each other only if $j\geq j_0$ for some $j_0\in\Z$ which only depends on $d$, $n$ and $\mu$. Plugging (\ref{Msharp 19}) into (\ref{Msharp 18}), we get
\begin{equation}\label{Msharp 20}
\begin{split}
\sum_{\begin{subarray}{c}Q\in\DD^\mu:\\Q\supset D\end{subarray}}
\sum_{R\in V(Q)}&\sum_{\begin{subarray}{c}P\in\Stop(R)\end{subarray}}
\left(\frac{\ell(P)}{\ell(R)}\right)^{1/2}\ell(R)^{-n}\left(|f_2|^2\mu\right)\left(P\right)\\
&\lesssim\sum_{\begin{subarray}{c}S\in\Top\end{subarray}}\sum_{\begin{subarray}{c}P\in\Stop(S)\end{subarray}}\left(\frac{\ell(P)}{D(P,D)}\right)^{n+1/2}
\frac{\left(|f_2|^2\mu\right)(P)}{\ell(P)^n}.
\end{split}
\end{equation}
Finally, by (\ref{Msharp 17}), (\ref{Msharp 20}), and (\ref{Msharp 16}), we conclude that
\begin{equation}\label{Msharp 21}
\begin{split}
S_1^2\lesssim\MM_2 f(x_0)^2
+\sum_{\begin{subarray}{c}S\in\Top\end{subarray}}\sum_{\begin{subarray}{c}P\in\Stop(S)\end{subarray}}\left(\frac{\ell(P)}{D(P,D)}\right)^{n+1/2}m_P\left(|f|^2\right).
\end{split}
\end{equation}

As we pointed out before, the same estimate holds for $S_2^2$, because the only properties that we used from the annuli $A^1_{m}(z_D)$'s are that they have bounded overlap for $m\in J_1$, that their width is comparable to $\ell(D)$, that they are centered in some point lying in $D\subset Q$ and that they have diameter comparable to $\ell(Q)$. Of course, these properties are also shared by the annuli $A^1_{m+1}(z_D)$'s. Actually, for estimating $S_2$, one can argue exactly as in the case of $S_1$ but replacing $\{m\in J_1:\, \epsilon_{m}\in I_Q\}$  by $\{m\in J_1:\, \epsilon_{m+1}\in I_Q\}$ in the involved arguments.
Therefore, by (\ref{Msharp 21}), the analogous estimate for $S_2$, and (\ref{Msharp 8}), we see that
\begin{equation}\label{Msharp 22}
\begin{split}
\bigg(\sum_{m\in J_1}b_m^\rho\bigg)^{1/\rho}
\lesssim\MM_2 f(x_0)
+\bigg(\sum_{\begin{subarray}{c}S\in\Top\end{subarray}}\sum_{\begin{subarray}{c}P\in\Stop(S)\end{subarray}}\left(\frac{\ell(P)}{D(P,D)}\right)^{n+1/2}m_P\left(|f|^2\right)\bigg)^{1/2}.
\end{split}
\end{equation}

We now deal with the sum relative to the $b_m$'s for $m\in J_2$. The estimates are essentially as in the case of $m\in J_1$, but we include the sketch of the arguments for the reader's convenience. Since $\rho>2$, (\ref{Msharp 7}) yields
\begin{equation}\label{Msharp 23}
\begin{split}
\bigg(\sum_{m\in J_2}b_m^\rho\bigg)^{1/\rho}
&\lesssim\Bigg(\sum_{m\in J_2}\bigg(
\int_{A^2_m(z_D)}|K(z_D-y)||f_2(y)|\,d\mu(y)\bigg)^2\Bigg)^{1/2}\\
&\quad+\Bigg(\sum_{m\in J_2}\bigg(
\int_{A^2_{m}(x)}|K(z_D-y)||f_2(y)|\,d\mu(y)\bigg)^2\Bigg)^{1/2}\\
&=:S_3+S_4.
\end{split}
\end{equation}
The arguments to estimate $S_3$ and $S_4$ are almost the same, so we will only give the details for $S_3$. Since $f_2$ vanishes in $3D$,
\begin{equation}\label{Msharp 24}
\begin{split}
S_3^2&=\sum_{k\in\Z}\sum_{\begin{subarray}{c}m\in J_2:\\ \epsilon_m\in I_k\end{subarray}}\!\bigg(
\int_{A^2_m(z_D)}\!|K(z_D-y)||f_2(y)|\,d\mu(y)\bigg)^2\!\!
\lesssim\sum_{\begin{subarray}{c}Q\in\DD^\mu:\\Q\supset D\end{subarray}}
\sum_{\begin{subarray}{c}m\in J_2:\\ \epsilon_m\in I_Q\end{subarray}}
\!\frac{\left|(|f_2|\mu)\left(A^2_m(z_D)\right)\right|^2}{\ell(Q)^{2n}}.\!
\end{split}
\end{equation}

Once again, our task now is to estimate $\left|(|f_2|\mu)\left(A^2_m(z_D)\right)\right|^2$. As before, this is done by splitting the annulus $A^2_m(z_D)$, whose width is $\epsilon_{m}-\epsilon_{m+1}$, in disjoint cubes $P\in\DD^\mu$ such that $\epsilon_{m}-\epsilon_{m+1}\in I_P$ and grouping them properly in terms of the corona decomposition. Arguing as in (\ref{Msharp 10}), we now have
\begin{equation}\label{Msharp 25}
\begin{split}
\left|(|f_2|\mu)\left(A^2_m(z_D)\right)\right|^2
&\lesssim\sum_{R\in V(Q)}\bigg|\sum_{\begin{subarray}{c}P\in\Tree(R):\\ \epsilon_m-\epsilon_{m+1}\in I_P\end{subarray}}(|f_2|\mu)\left(A^2_m(z_D)\cap P\right)\bigg|^2\\
&\quad+\sum_{R\in V(Q)}\bigg|\sum_{\begin{subarray}{c}P\in\Stop(R):\\ \ell(P)\geq\epsilon_m-\epsilon_{m+1}\end{subarray}}(|f_2|\mu)\left(A^2_m(z_D)\cap P\right)\bigg|^2.
\end{split}
\end{equation}
The first term on the right hand side of (\ref{Msharp 25}) can be easily estimated using Cauchy-Schwarz's inequality, that the $P$'s in $\Tree(R)$ such that $\epsilon_m-\epsilon_{m+1}\in I_P$ are disjoint and Lemma \ref{annulus}. Similarly to what we did in (\ref{Msharp 11}), we now obtain
\begin{equation}\label{Msharp 26}
\begin{split}
&\bigg|\sum_{\begin{subarray}{c}P\in\Tree(R):\\ \epsilon_m-\epsilon_{m+1}\in I_P\end{subarray}}(|f_2|\mu)\left(A^2_m(z_D)\cap P\right)\bigg|^2\\
&\qquad\qquad\qquad\lesssim(\epsilon_m-\epsilon_{m+1})\ell(R)^{n-1}\sum_{\begin{subarray}{c}P\in\Tree(R):\\ \epsilon_m-\epsilon_{m+1}\in I_P\end{subarray}}\left(|f_2|^2\mu\right)\left(A^2_m(z_D)\cap P\right)\\
&\qquad\qquad\qquad\leq\ell(D)\ell(R)^{n-1}\left(|f_2|^2\mu\right)\left(A^2_m(z_D)\cap R\right),
\end{split}
\end{equation}
where we also used in the last inequality above  that $\epsilon_m-\epsilon_{m+1}\leq\ell(D)$, because we are assuming $m\in J_2$.
As before, the second term on the right hand side of (\ref{Msharp 25}) is estimated similarly to (\ref{Msharp 26}) but introducing an auxiliary splitting of the sum in terms of the side length of the cubes. By applying the Cauchy-Schwarz inequality, we can proceed exactly as in (\ref{Msharp 12}) and (\ref{Msharp 13}), but replacing $\ell(D)$ by $\epsilon_m-\epsilon_{m+1}$, and then we deduce that
\begin{equation}\label{Msharp 27}
\begin{split}
&\bigg|\sum_{\begin{subarray}{c}P\in\Stop(R):\\ \ell(P)\geq\epsilon_m-\epsilon_{m+1}\end{subarray}}
(|f_2|\mu)\left(A^1_m(z_D)\cap P\right)\bigg|^2\\
&\qquad\qquad\qquad\lesssim\sum_{\begin{subarray}{c}P\in\Stop(R)\end{subarray}}
\left(\frac{\ell(P)}{\ell(R)}\right)^{1/2}\ell(R)^n\left(|f_2|^2\mu\right)\left(A^2_m(z_D)\cap P\right).
\end{split}
\end{equation}

Combining (\ref{Msharp 24}) and (\ref{Msharp 25}) with (\ref{Msharp 26}) and (\ref{Msharp 27}), and using that $\ell(R)=\ell(Q)$ for all $R\in V(Q)$ and that, for $m\in\Z$, the closed annuli $A^2_m(z_D)$'s overlap  only with the neighboring annuli because  $\{\epsilon_m\}_{m\in\Z}$ is a non-increasing sequence, we conclude that
\begin{equation}\label{Msharp 28}
\begin{split}
S_3^2&\lesssim
\sum_{\begin{subarray}{c}Q\in\DD^\mu:\\Q\supset D\end{subarray}}
\sum_{R\in V(Q)}
\frac{\ell(D)}{\ell(R)}\,\ell(R)^{-n}\left(|f_2|^2\mu\right)\left(R\right)\\
&\quad+\sum_{\begin{subarray}{c}Q\in\DD^\mu:\\Q\supset D\end{subarray}}
\sum_{R\in V(Q)}\sum_{\begin{subarray}{c}P\in\Stop(R)\end{subarray}}
\left(\frac{\ell(P)}{\ell(R)}\right)^{1/2}\ell(R)^{-n}\left(|f_2|^2\mu\right)\left(P\right).
\end{split}
\end{equation}
Plugging (\ref{Msharp 17}) and (\ref{Msharp 20}) into (\ref{Msharp 28}) finally yields
\begin{equation}\label{Msharp 29}
\begin{split}
S_3^2\lesssim\MM_2 f(x_0)^2
+\sum_{\begin{subarray}{c}S\in\Top%:\\3S\supset D
\end{subarray}}\sum_{\begin{subarray}{c}P\in\Stop(S)\end{subarray}}\left(\frac{\ell(P)}{D(P,D)}\right)^{n+1/2}m_P\left(|f|^2\right).
\end{split}
\end{equation}

Similarly to what we said below (\ref{Msharp 21}), the same estimate that we have for $S_3$ also holds for $S_4$. Therefore, applying (\ref{Msharp 29}) (and the same estimate for $S_4$) to (\ref{Msharp 23}), we see that
\begin{equation}\label{Msharp 30}
\begin{split}
\bigg(\sum_{m\in J_2}b_m^\rho\bigg)^{1/\rho}
\lesssim\MM_2 f(x_0)
+\bigg(\sum_{\begin{subarray}{c}S\in\Top\end{subarray}}\sum_{\begin{subarray}{c}P\in\Stop(S)\end{subarray}}\left(\frac{\ell(P)}{D(P,D)}\right)^{n+1/2}m_P\left(|f|^2\right)\bigg)^{1/2}.
\end{split}
\end{equation}

To complete the proof of the theorem it only remains to put all the estimates together and to use standard arguments. From
$(\ref{Msharp 6})$, $(\ref{Msharp 22})$ and $(\ref{Msharp 30})$, we see that
\begin{equation*}
\begin{split}
\bigg(\sum_{m\in \Z}(a_m+b_m)^\rho\bigg)^{1/\rho}
\lesssim\MM_2 f(x_0)
+\bigg(\sum_{\begin{subarray}{c}S\in\Top\end{subarray}}\sum_{\begin{subarray}{c}P\in\Stop(S)\end{subarray}}\left(\frac{\ell(P)}{D(P,D)}\right)^{n+1/2}m_P\left(|f|^2\right)\bigg)^{1/2},
\end{split}
\end{equation*}
which, by (\ref{Msharp 4}) and (\ref{Msharp 5}), implies that
\begin{equation}\label{Msharp 31}
\begin{split}
I_2&=\frac{1}{\mu(D)}\int_D|(\VV_\rho\circ\TT^\mu)f_2-(\VV_\rho\circ\TT^\mu) f_2(z_D)|\,d\mu\\
&\lesssim\MM_2 f(x_0)
+\bigg(\sum_{\begin{subarray}{c}S\in\Top\end{subarray}}\sum_{\begin{subarray}{c}P\in\Stop(S)\end{subarray}}\left(\frac{\ell(P)}{D(P,D)}\right)^{n+1/2}m_P\left(|f|^2\right)\bigg)^{1/2}.
\end{split}
\end{equation}
Finally, combining (\ref{Msharp 1}) and (\ref{Msharp 2}) with (\ref{Msharp 3}) and (\ref{Msharp 31}), and using that $\bigcup_{S\in\Top}\Stop(S)\subset\Top$, we conclude that
\begin{equation}\label{Msharp 32}
\begin{split}
(\MM^\sharp_{\DD^\mu}\circ\VV_\rho&\circ\TT^\mu)f(x_0)\\
&\lesssim\MM_2 f(x_0)+\sup_{\begin{subarray}{c}D\in\DD^\mu:\,x_0\in D\end{subarray}}
\bigg(\sum_{\begin{subarray}{c}S\in\Top\end{subarray}}\sum_{\begin{subarray}{c}P\in\Stop(S)\end{subarray}}\left(\frac{\ell(P)}{D(P,D)}\right)^{n+1/2}m_P\left(|f|^2\right)\bigg)^{1/2}\\
&\lesssim\MM_2 f(x_0)
+\bigg(\sum_{\begin{subarray}{c}P\in\Top\end{subarray}}\left(\frac{\ell(P)}{D(P,x_0)}\right)^{n+1/2}m_P\left(|f|^2\right)\bigg)^{1/2}\\
&=:\MM_2 f(x_0)+\EE_{1/2} f(x_0),
\end{split}
\end{equation}
for all $x_0\in\supp(\mu)$, where we denoted
\begin{equation}\label{eqDD2}
D(P,x_0):=\ell(P)+ \dist(P,x_0).
\end{equation}

In Lemma \ref{lemma E} below we prove that $\EE_{1/2}$ is a bounded operator in $L^p(\mu)$ for all $2<p<\infty$. Assuming this for the moment, by (\ref{Msharp 32}) and the $L^p(\mu)$-boundedness of $\MM_2$, we see that $\MM^\sharp_{\DD^\mu}\circ\VV_\rho\circ\TT^\mu$ is also bounded in $L^p(\mu)$ for all $2<p<\infty$.  Then we obtain
\begin{equation*}
\begin{split}
\|(\VV_\rho\circ\TT^\mu)f\|_{L^p(\mu)}
\leq\|(\MM_{\DD^\mu}\circ\VV_\rho\circ\TT^\mu)f\|_{L^p(\mu)}
\lesssim\|(\MM^\sharp_{\DD^\mu}\circ\VV_\rho\circ\TT^\mu)f\|_{L^p(\mu)}
\lesssim\|f\|_{L^p(\mu)}
\end{split}
\end{equation*}
for all $2<p<\infty$, and the theorem is proved.
\end{proof}

\vspace{2mm}

\begin{lemma}\label{lemma E}
Given $\delta>0$, set
$$\EE_{\delta} f(x):=\bigg(\sum_{\begin{subarray}{c}P\in\Top\end{subarray}}\left(\frac{\ell(P)}{D(P,x)}\right)^{n+\delta}m_P\left(|f|^2\right)\bigg)^{1/2}$$
for $f\in L^p(\mu)$ and $x\in\R^d$, where $D(P,x)$ is defined in \eqref{eqDD2}.
Then $\EE_{\delta}$ is a bounded operator in $L^p(\mu)$ for all $2<p<\infty$.
\end{lemma}
\begin{proof}
The proof follows by duality and Carleson's embedding theorem. Since $2<p<\infty$, if $q$ is such that $2/p+1/q=1$ then $1<q<\infty$, thus
\begin{equation}\label{E eq4}
\begin{split}
\|\EE_{\delta} f\|_{L^p(\mu)}=\|(\EE_{\delta} f)^2\|_{L^{p/2}(\mu)}^{1/2}
=\sup_{\|g\|_{L^{q}(\mu)}\leq1}\left|\int(\EE_{\delta} f)^2g\,d\mu\right|^{1/2}.
\end{split}
\end{equation}
Note that
\begin{equation}\label{E eq1}
\begin{split}
\left|\int(\EE_{\delta} f)^2g\,d\mu\right|
\leq\sum_{\begin{subarray}{c}P\in\Top\end{subarray}}m_P\left(|f|^2\right)\int\left(\frac{\ell(P)}{D(P,x)}\right)^{n+\delta}|g(x)|\,d\mu(x).
\end{split}
\end{equation}
Integrating over dyadic annuli and using that $\mu$ is $n$-AD regular, it is easy to check that
\begin{equation}\label{E eq2}
\begin{split}
\frac{1}{\mu(P)}\int\left(\frac{\ell(P)}{D(P,x)}\right)^{n+\delta}|g(x)|\,d\mu(x)\lesssim\MM g(y)\quad \mbox{ for all $y\in P$}
\end{split}
\end{equation}
 (here it is crucial that $\delta>0$). Thus, by (\ref{E eq1}), (\ref{E eq2}), H\"older's inequality and Carleson's embedding Theorem \ref{teocarleson} (recall that $p/2$ and $q$ belong to $(1,\infty)$),
\begin{equation}\label{E eq3}
\begin{split}
\left|\int(\EE_{\delta} f)^2g\,d\mu\right|
&\lesssim\sum_{\begin{subarray}{c}P\in\Top\end{subarray}}m_P\left(|f|^2\right)m_P(\MM g)\mu(P)\\
&\leq\bigg(\sum_{\begin{subarray}{c}P\in\Top\end{subarray}}\bigl(m_P\left(|f|^2\right)\bigr)^{p/2}\mu(P)\bigg)^{2/p}
\bigg(\sum_{\begin{subarray}{c}P\in\Top\end{subarray}}(m_P(\MM g))^{q}\mu(P)\bigg)^{1/q}\\
&\lesssim\||f|^2\|_{L^{p/2}(\mu)}\|\MM g\|_{L^{q}(\mu)}
\lesssim\|f\|_{L^{p}(\mu)}^2\|g\|_{L^{q}(\mu)}.
\end{split}
\end{equation}
From (\ref{E eq4}) and  (\ref{E eq3}) we conclude that
$\|\EE_{\delta} f\|_{L^{p}(\mu)}\lesssim\|f\|_{L^{p}(\mu)}$, as wished.
\end{proof}

\section{The proof of Theorem \ref{teopra}}\label{sec5}

The arguments are very similar to the ones for the proof of Theorem \ref{teopri} and so we will only sketch
the main ideas.

When $K$ is an odd kernel satisfying \eqref{4eq333}, one of the main ingredients of the proof of the boundedness of $\VV_\rho\circ\TT$ from $M(\R^d)$ to $L^{1,\infty}(\mu)$ in Section \ref{sec3}
and of  $\VV_\rho\circ\TT^\mu$ in $L^p(\mu)$ for $2<p<\infty$ in Section \ref{sec4} is Theorem
\ref{weakL1 thm smooth}, which ensures the boundedness of
 $\VV_\rho\circ\TT^\mu $ in $L^{2}(\mu)\to L^{2}(\mu)$  and of
 $\VV_\rho\circ\TT_\varphi$ from $M(\R^d)$ to $L^{1,\infty}(\mu)$. The reader can easily check that
 exactly the same arguments contained in Sections \ref{sec3} and \ref{sec4} show that
if $K(\cdot,\cdot)$ is a Calder\'on-Zygmund kernel as in Theorem \ref{teopra} and $T$ is the associated operator, and moreover
the following assumptions hold:
\begin{itemize}
\item [$(i)$] $\VV_\rho\circ\TT^\mu:L^{2}(\mu)\to L^{2}(\mu)$  is bounded,
\item [$(ii)$] $\VV_\rho\circ\TT_\varphi:M(\R^d)\to L^{1,\infty}(\mu)$  is bounded,
\end{itemize}
then
 $\VV_\rho\circ\TT:M(\R^d)\to L^{1,\infty}(\mu)$ and
  $\VV_\rho\circ\TT^\mu:L^{p}(\mu)\to L^{p}(\mu)$, $2<p<\infty$, are also bounded. That is, the same conclusions of Theorems \ref{weakL1 thm sl} and \ref{Lp thm} hold.

Thus, by interpolation, to conclude the proof of Theorem \ref{teopra} it just remains to check that the conditions $(i)$
and $(ii)$ above hold. This is obvious in the case of condition $(i)$ because this is indeed one of the
 main assumptions of Theorem \ref{teopra}. Concerning $(ii)$, note first that the boundedness of
 $\VV_\rho\circ\TT^\mu$ in $L^{2}(\mu)$ implies that $\VV_\rho\circ\TT^\mu_\varphi$ is also bounded in $L^{2}(\mu)$. This is an immediate consequence of the pointwise estimate
$$\VV_\rho\circ\TT^\mu_\varphi(f)(x)\lesssim \VV_\rho\circ\TT^\mu(f)(x),$$
which can be obtained by writing
$$T_{\varphi_\epsilon}(f\mu)(x) :=  \int \varphi_\epsilon(x-y)K(x,y)\,f(y)\,d\mu(y)$$
in terms of a convex combination of functions of the form
$$T_{\delta}(f\mu)(x) :=  \int_{|x-y|>\delta}K(x,y)\,f(y)\,d\mu(y),$$
for  $\delta>0$ belonging to some interval depending on $\epsilon$ and then applying Minkowski's integral inequality. The arguments are quite similar to the ones
in \eqref{eq2}-\eqref{eq4} and we omit them.

Then, basically the same arguments for the proof of Theorem 2.5 in \cite{MT} show that the boundedness
of $\VV_\rho\circ\TT^\mu_\varphi$ in $L^2(\mu)$ implies that $\VV_\rho\circ\TT_\varphi$ is bounded
from $M(\R^d)$ to $L^{1,\infty}(\mu)$. This is shown in \cite{MT} for the case when
$K$ is an odd kernel satisfying \eqref{4eq333} and $\mu$ is the
Hausdorff measure $\HH^n$ on a Lipschitz graph. However, the same proof with very minor changes works
in the more general situation when $K(\cdot,\cdot)$ is a kernel such as in Theorem \ref{teopra} and
$\mu$ is just and $n$-dimensional AD-regular measure.

\end{document}